\documentclass[12pt]{article}

\usepackage{graphicx,tikz,color,url}
\usepackage{amsmath,amssymb,latexsym}
\usepackage{pgfplots}
\usepackage{mathrsfs}
\usepackage{cite}
\usepackage[hidelinks]{hyperref}
\usepackage{soul}
\usepackage{caption}
\usepackage{subcaption}
\usepackage{multicol}
\usepackage{mathdots}
\usepackage{cite}
\usepackage{mathtools}
\usepackage{adjustbox}
\usepackage{float}
\usepackage{array}
\usepackage{setspace}
\usepackage{tikz-timing}
\usepackage{enumitem}
\usetikzlibrary{positioning,shapes,fit,arrows}
\usetikzlibrary{graphs}
\usetikzlibrary{graphs.standard}
\usetikzlibrary{shadows}
\definecolor{myblack}{RGB}{0,0,0}
\definecolor{white}{rgb}{1.0, 1.0, 1.0}
\definecolor{red}{rgb}{1.0, 0.0, 0.0}

\newtheorem{theorem}{Theorem}[section]

\newtheorem{lemma}[theorem]{Lemma}
\newtheorem{proposition}[theorem]{Proposition}
\newtheorem{corollary}[theorem]{Corollary}

\newtheorem{observation}[theorem]{Observation}
\newtheorem{remark}{Remark}

\newcommand{\proof}{\noindent{\bf Proof.\ }}
\newcommand{\qed}{\hfill $\square$ \bigskip}

\newcommand{\ie}{that is}

\begin{document}

\title{Packing chromatic critical graphs with radius at most $2$}

\author{
Asl{\i}han G\"{u}r$^{a,}$\thanks{Email: \texttt{agur@gtu.edu.tr}}
\and 
 Didem G\"{o}z\"{u}pek$^{b,}$\thanks{Email: \texttt{didem.gozupek@gtu.edu.tr}}
\and 
Hadi Alizadeh$^{a,}$\thanks{Email: \texttt{halizadeh@gtu.edu.tr}}
}
\maketitle

\begin{center}
$^a$ Department of Mathematics, Gebze Technical University, T\"urkiye\\
	\medskip
	$^b$ Department of Computer Engineering, Gebze Technical University, T\"urkiye
\end{center}

\begin{abstract}
For a graph $G$ with vertex set $V(G)$ and a positive integer $i$, an $i$-packing in $G$ is a subset $X$ of $V(G)$ such that the distance between any two distinct vertices of $X$ is greater than $i$. 
The packing chromatic number of $G$, denoted by $\chi_{\rho}(G)$, is the smallest positive integer $k$ for which there exists a partition $X_1, X_2, \ldots, X_k$ of $V(G)$ such that $X_i$ is an $i$-packing in $G$ for every $i \in [k]$. A graph $G$ is called  $\chi_\rho$-critical if $\chi_\rho(H) < \chi_\rho(G)$ holds for every proper subgraph $H$ of $G$. In this paper, we provide a structural characterization of $\chi_{\rho}$-critical graphs with radius $1$, and completely determine the $\chi_{\rho}$-critical cactus graphs with radius $2$ and diameter $2$ or $3$.
\end{abstract}


\noindent
{\bf AMS Subj.\ Class.\ (2020)}: 05C15, 05C12, 05C70, 05C75.

%%%%%%%%%%%%%%%%%%%%%%%%%%%%%%%%%%%%%%%
\section{Introduction} \label{sec:intro}
%%%%%%%%%%%%%%%%%%%%%%%%%%%%%%%%%%%%%%%
The concept of \textit{packing coloring} was originally introduced by Goddard et al.~\cite{goddard-2008} under the name \textit{broadcast coloring}. Bre\v{s}ar et al.~\cite{bresar-2007} used the term packing coloring for the same concept.

For a graph $G$ with vertex set $V(G)$ and a positive integer $i$, an $i$-packing in $G$ is a subset $X$ of $V(G)$ such that $d_G(u,v)>i$ for every two distinct vertices $u,v \in X$.
\textit{The packing chromatic number} of $G$, denoted by $\chi_{\rho}(G)$, is the smallest positive integer $k$ for which there exists a partition $X_1, X_2, \ldots, X_k$ of $V(G)$ such that $X_i$ is an $i$-packing in $G$ for every $i \in [k]$.
Equivalently, a partition $X_1, X_2, \ldots, X_k$ of $V(G)$ defines a coloring $c : V(G) \to [k]$ by setting $c(u)=i$ for every  $u \in X_i$. This coloring is called a $k$-packing coloring of $G$.

A graph $G$ is called \textit{packing chromatic critical} or $\chi_\rho$-critical if $\chi_\rho(H) < \chi_\rho(G)$ for every proper subgraph $H$ of $G$ and $\chi_\rho(G) = k$, then $G$ is called a $k$-$\chi_\rho$-critical graph~\cite{bresar-2022}.
If $\chi_\rho(G - v) < \chi_\rho(G)$ for every $v \in V(G)$, then $G$ is called $\chi_\rho$-vertex-critical~\cite{klavzar-2019}. Similarly, if $\chi_\rho(G - e) < \chi_\rho(G)$ for every $e \in E(G)$, then $G$ is called $\chi_\rho$-edge-critical; for graphs with no isolated vertices, this is equivalent to $\chi_\rho$-criticality~\cite{bresar-2022}.

Goddard et al.~\cite{goddard-2008} determined the packing chromatic number of paths and cycles, characterized graphs with $\chi_{\rho}(G) \in \{2,3\}$, and obtained several results for trees.
Bre\v{s}ar et al.~\cite{bresar-2018}, established several bounds related to the independence number $\alpha(G)$ and provided necessary conditions for graphs satisfying $\chi_{\rho}(G)=\chi(G)$.
Goddard et al.~\cite{goddard-2008} showed that deciding whether a planar graph admits a $4$-packing coloring is NP-hard. The packing coloring problem is also NP-complete for trees~\cite{fiala-2010} and for chordal graphs of diameter at least three~\cite{kim-2018}.

Numerous results have also been obtained for specific graph classes, such as, trees, subcubic graphs, Sierpi\'nski-type graphs, Cartesian products of graphs,   lexicographic products of graphs, corona of graphs, and infinite graphs~\cite{goddard-2008, rall-2008, bresar-2007, balogh-2018, bresar-ferme-2018-1, gastineau-2019, bresar-gastineau-2020, deng-2021, bresar-ferme-2018, korze-2019, bozovic-2021, laiche-2017, fiala-2009, togni-2014, ekstein-2014, martin-2017}. For a comprehensive overview, we refer to the survey~\cite{bresar-survey-2020}. Packing coloring continues to be actively studied; see, for example,~\cite{furmanczyk-2025,gregor-2024}.

We now turn to packing chromatic critical graphs. Klav\v{z}ar and Rall~\cite{klavzar-2019} introduced $\chi_{\rho}$-vertex-critical graphs and established several structural results. In particular, they characterized $3$-$\chi_{\rho}$-vertex-critical graphs and provided a partial characterization for $4$-$\chi_{\rho}$-vertex-critical graphs.
Subsequently, Ferme characterized all $4$-$\chi_{\rho}$-vertex-critical graphs in~\cite{ferme-2022}. 
In~\cite{bresar-2022}, Bre\v{s}ar and Ferme introduced $\chi_{\rho}$-critical graphs, and characterized the $2$-$\chi_{\rho}$-critical and $3$-$\chi_{\rho}$-critical graphs, proved that for trees $\chi_{\rho}$-criticality is equivalent to $\chi_{\rho}$-vertex-criticality, and gave characterizations of $\chi_{\rho}$-critical graphs with diameter $2$, as well as of $\chi_{\rho}$-critical block graphs with diameter $3$. They posed the open question of characterizing $\chi_\rho$-critical graphs of radius $2$. Motivated by the latter two results, we address this question for cactus graphs by characterizing all $\chi_{\rho}$-critical cactus graphs of radius $2$ and diameter $2$ or $3$. We also obtain a structural characterization of $\chi_{\rho}$-critical graphs of radius $1$.

The remainder of this paper is organized as follows. In Section~\ref{sec:pre}, we present basic definitions and notations used throughout the paper. In Section~\ref{sec:rad1}, we present a structural characterization of $\chi_{\rho}$-critical graphs with radius $1$. In Section~\ref{sec:rad2}, we determine all $\chi_{\rho}$-critical cactus graphs with radius $2$ and diameter $2$ or $3$. Sections~\ref{sec:rad22} and~\ref{sec:rad23} consider these two cases, respectively. Finally, in Section~\ref{sec:mainthm} we provide a complete structural characterization for $\chi_{\rho}$-critical cactus graphs with radius $2$ and diameter $3$, we summarize our main results and present some open problems.

%%%%%%%%%%%%%%%%%%%%%%%%%%%%%%%%%%%%%%%
\section{Preliminaries} \label{sec:pre}
%%%%%%%%%%%%%%%%%%%%%%%%%%%%%%%%%%%%%%%

Throughout this paper, all graphs are finite and simple. Let $G = (V(G), E(G))$ be a graph with vertex set $V(G)$ and edge set $E(G)$, and let $u, v \in V(G)$ be arbitrary vertices.
We write $G-u$ and $G-e$ for the subgraphs obtained by removing a vertex $u$ and an edge $e$ from $G$, respectively. 
The \emph{distance} between $u$ and $v$, denoted by $d_G(u,v)$, is the length of a shortest $u$--$v$ path in $G$.
The \emph{neighborhood} of $u$, denoted by $N_G(u)$, is the set of vertices adjacent to $u$. If $v\in N_G(u)$, then $u$ and $v$ are called \emph{neighbors} in $G$. 

For $S \subseteq V(G)$, the subgraph induced by $S$ is the graph $G[S]$ with vertex set $S$ and edge set consisting of the edges of $G$ with both endpoints in $S$. The \emph{eccentricity} of  \(u\), denoted by \(\varepsilon_G(u)\), 
is the maximum distance between \(u\) and any other vertex of \(G\), \ie,
$\varepsilon_G(u) = \max_{v \in V(G)} \{\mathrm{d}_G(u,v)\}.$
The \emph{radius} of $G$ is the minimum eccentricity of its vertices, and the \emph{diameter} of $G$ is the maximum eccentricity of its vertices, \ie,
$\operatorname{rad}(G) = \min_{u \in V(G)} \varepsilon_G(u)$ and $\operatorname{diam}(G) = \max_{u \in V(G)} \varepsilon_G(u).$ 
The \emph{center} of $G$ is the set of all vertices $v$ such that the eccentricity of $v$ equals the radius of $G$.

A vertex $u$ is \emph{universal} if $d_G(u,v) = 1$ for all $v \in V(G) \setminus \{u\}$. 
A \emph{leaf} is a vertex adjacent to exactly one vertex in $G$. 
A vertex $u$ is a \emph{cut vertex} if $G - u$ is disconnected, and an edge $e$ is a \emph{cut edge} if $G - e$ is disconnected. The \emph{order} of a graph $G$ is its number of vertices. A path, a cycle, and a complete graph of order $n$ are denoted by $P_n$, $C_n$, and $K_n$, respectively, and a star of order $n+1$ is denoted by $K_{1,n}$.

A \emph{block} is a connected graph with no cut vertex. 
A graph $G$ is called a \emph{cactus graph} if it is connected and each block is either a cycle or a $K_2$. 
If $G$ is acyclic, then $G$ is a tree; if $G$ has a cycle $C$, then $C$ is chordless, since each edge belongs to at most one cycle.

A set $S \subseteq V(G)$ is called an \emph{independent set} if no two vertices in $S$ are adjacent. The \emph{independence number} of $G$, denoted by $\alpha(G)$, is the maximum size of an independent set in $G$. 
A graph $G$ is called an \emph{$\alpha$-critical graph} if $\alpha(G - e) > \alpha(G)$ for every edge $e \in E(G)$. 
For further definitions and basic concepts in graph theory, we refer the reader to~\cite{west-1996}.

We begin with a simple observation. 

\begin{observation}\label{obsv1}
If $H_1$ and $H_2$ are the connected components of $G$ after removing a cut edge $e$, then $\operatorname{diam}(G) \geq \operatorname{diam}(H_i)$ for $i \in \{1, 2\}$.
\end{observation}

%%%%%%%%%%%%%%%%%%%%%%%%%%%%%%%%%%%%%%%
\section{Structural characterization of $\chi_{\rho}$-critical graphs with radius $1$}
\label{sec:rad1}

In this section, we provide a complete structural characterization of \( \chi_\rho \)-critical graphs with \( \mathrm{rad}(G) = 1 \). If \( \mathrm{rad}(G) = 1 \), then $G$ contains a universal vertex. It is well known that for any connected graph $G$,
$\operatorname{rad}(G) \le \operatorname{diam}(G) \le 2\,\operatorname{rad}(G)$. Therefore, $\operatorname{diam}(G) \in \{1,2\}$.

First we consider the case \( \mathrm{rad}(G) = 1 \) and \( \mathrm{diam}(G) = 1 \).

\begin{observation}\label{obs:complete}
Let $G$ be a graph with $\operatorname{rad}(G)=1$ and $\operatorname{diam}(G)=1$. 
Then $G \cong K_n$ with $n \ge 2$  and $G$ is $\chi_\rho$-critical.
\end{observation}

\proof
Let $G$ be a graph with $\operatorname{rad}(G)=1$ and $\operatorname{diam}(G)=1$. Since $\operatorname{diam}(G)=1$, $G$ is a complete graph, \ie, $G \cong K_n$ with $n \ge 2$. 
It is well known that complete graphs are $\chi_\rho$-critical.
\qed

We therefore consider the case $\mathrm{rad}(G) = 1$ and $\mathrm{diam}(G) = 2$. Bre\v{s}ar and Ferme \cite{bresar-2022} characterized $\chi_\rho$-critical graphs of diameter $2$. For graphs with radius $1$ and diameter $2$, we reformulate this criterion in structural terms. In particular, we show that if $u$ is a universal vertex of $G$, then the $\chi_\rho$-criticality of $G$ is determined by the structure of $G-u$. 

We begin by recalling the following result of Goddard et al.~\cite{goddard-2008}, which we use frequently.

\begin{lemma}~{\rm \cite{goddard-2008}}\label{lemma4}
For every graph $G$, $\chi_\rho(G) \leq |V(G)|-\alpha(G)+1,$ with equality if $G$ has diameter two.
\end{lemma}

Using Lemma \ref{lemma4}, we establish the following result.

\begin{proposition}\label{pro14}
If $G$ is a $\chi_\rho$-critical graph with at least three vertices and $\mathrm{rad}(G) = 1$, then $G$ has no leaf.
\end{proposition}

\proof Let $G$ be a $\chi_\rho$-critical graph with at least three vertices and $\mathrm{rad}(G) = 1$.  If $\mathrm{diam}(G) = 1$, then it is clear by Observation \ref{obs:complete}.

Let $\mathrm{diam}(G) = 2$.  Suppose to the contrary that $G$ has a leaf $v$. Then $G$ has a unique universal vertex $u$ and $v$ is adjacent only to $u$. Since the diameter of $G$ is $2$, it is not a complete graph and hence, $\alpha(G) \geq 2$. So, there is no maximum independent set in $G$ that contains the universal vertex $u$, \ie, every maximum independent set in $G$ contains $v$. Since $G$ is $\chi_\rho$-critical, we have $\chi_\rho(G-e)<\chi_\rho(G)$, where $e=u v$ is the edge between $u$ and $v$. After removing $e$, $G-e$ becomes disconnected with connected components $H_1$ and $H_2$, where $V\left(H_2\right)=\{v\}$. Since $v$  belongs to every maximum independent set of $G$, $\alpha\left(H_1\right)=\alpha(G)-1$.
Additionally, 
$\chi_\rho(G-e)=\max \left(\chi_\rho\left(H_1\right), \chi_\rho\left(H_2\right)\right)=\chi_\rho\left(H_1\right).$
 By Observation \ref{obsv1} the diameter of $H_1$ satisfies $\operatorname{diam}\left(H_1\right) \leq 2$ since $e$ is a cut edge.

If $\operatorname{diam}\left(H_1\right)=1$, then $H_1$ is a complete graph. In this case,
$\chi_\rho(G-e)=\chi_\rho\left(H_1\right)=|V(H_1)|=|V(G)|-1.$
However, since $\alpha(G) = 2$, Lemma~\ref{lemma4} yields 
$\chi_\rho(G) = |V(G)| - \alpha(G) + 1 = |V(G)| - 1,$ since $G$ has diameter $2$.
This leads to $\chi_\rho(G-e)=\chi_\rho(G)$, contradicting with $G$ being  $\chi_\rho$-critical.

 If $\operatorname{diam}\left(H_1\right)=2$, then by Lemma \ref{lemma4} 
$\chi_\rho(G-e)=\chi_\rho\left(H_1\right)=\left|V\left(H_1\right)\right|-\alpha\left(H_1\right)+1.$
By replacing $\left|V\left(H_1\right)\right|=|V(G)|-1$ and 
$\alpha\left(H_1\right)=\alpha(G)-1$, we have
$\chi_\rho(G-e)=|V(G)|-\alpha(G)+1=\chi_\rho(G)$, which contradicts the assumption that $G$ is $\chi_\rho$-critical.

Since both cases result in contradictions, we conclude that $G$ does not have a leaf. \qed

The next result describes the structure of $\chi_\rho$-critical graphs with radius 1.

\begin{proposition}\label{cor1}
Let $G$ be a graph with at least three vertices and $\mathrm{rad}(G) = 1$. If $G$ is $\chi_\rho$-critical, then $G-u$ is a disjoint union of $\alpha$-critical graphs for every universal vertex $u \in V(G)$. 
\end{proposition}

\proof Let $G$ be a $\chi_\rho$-critical graph with at least three vertices and $\mathrm{rad}(G) = 1$.  If $\mathrm{diam}(G) = 1$, then it is clear by Observation \ref{obs:complete}.

Let  $\mathrm{diam}(G) = 2$ and $u$ be a universal vertex of $G$. 
Since $G$ is $\chi_\rho$-critical, it follows from Proposition \ref{pro14} that $G$ has no leaf. Then every connected component of $G - u$ contains at least two vertices, and therefore each component contains at least one edge.
Let $A_1, A_2, \dots, A_r$ be the connected components of $G-u$, where $r \geq 1$. 
Suppose to the contrary that there exists a connected component of $G-u$, say $A_1$, that is not $\alpha$-critical. This implies that there exists an edge $e \in E\left(A_1\right)$ such that $\alpha\left(A_1-e\right)=\alpha\left(A_1\right)$. In this case, $\alpha(G-e)=\alpha(A_1-e)+\sum_{i=2}^{r} \alpha(A_i)=\sum_{i=1}^{r} \alpha(A_i)$ and $\alpha(G)=\sum_{i=1}^{r} \alpha(A_i)$.
Since $e \in E(A_1)$, the diameter of $G-e$ remains $2$. Applying  Lemma \ref{lemma4}, we obtain the following contradiction.
\begin{align*}
\chi_\rho(G-e)&=|V(G-e)| - \alpha(G-e)  + 1
\\ &=|V(G)| - \sum_{i=1}^{r} \alpha(A_i) + 1 \\ &=|V(G)| - \alpha(G)  + 1\\ &=\chi_\rho(G).
\end{align*}
Thus, every connected component of $G-u$ is $\alpha$-critical. 
\qed

\begin{remark}
Let $G$ be a graph with $\mathrm{rad}(G)=1$ and $\mathrm{diam}(G)=2$. Even if $G-u$ is a disjoint union of $\alpha$-critical graphs for every universal vertex $u$ of $G$, $G$ need not be $\chi_\rho$-critical.
\end{remark}

For instance, consider the wheel graph $W_6$. Let $u$ be the universal vertex of $W_6$. Then $W_6-u \cong C_5$, which is $\alpha$-critical.
However, $W_6$ is not $\chi_\rho$-critical, since there exists an edge $e$ such that $\chi_\rho(W_6-e)=\chi_\rho(W_6)=5$ (see Figure~\ref{fig3}).

\begin{figure}[H]
\centering
\begin{tikzpicture}[scale=2,
vertex/.style={circle, draw=black, fill=white, thick, minimum size=4pt, inner sep=2pt}]
\foreach \i in {1,...,5}
\node[vertex] (P\i) at (90+72*\i:1) {};
		
% Center vertex
\node[vertex] (C) at (0,0) {};
		
% Edges of the outer cycle
\foreach \i [evaluate=\i as \j using {int(mod(\i,5)+1)}] in {1,...,5}
\draw[black, thick] (P\i) -- (P\j);
		
% Spokes to the center
\foreach \i in {1,...,5}
\draw[black, thick] (C) -- (P\i);
		
% Example edge e
\path (C) -- (P3) coordinate[pos=0.5] (mid);
\node at ([xshift=1pt, yshift=3pt] mid) {$e$};
\node at ([xshift=-4.5pt, yshift=16pt] mid) {$u$};
\end{tikzpicture}
\caption{The wheel graph $W_6$.}
\label{fig3}
\end{figure}
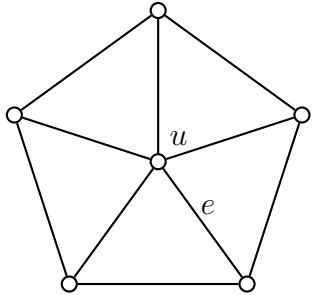

In Lemma \ref{thm:iff-universal-cut}, we characterize $\chi_\rho$-critical graphs with radius $1$ and diameter $2$ that have a universal cut vertex, \ie, a cut vertex adjacent to every other vertex.
Our proof will use a corollary of the next result due to Haynes et al.~\cite{haynes-1990}.

\begin{lemma}~{\rm \cite{haynes-1990}}\label{teo2}
A graph $G$ has $\alpha(G) \neq \alpha(G-e)$ for all  $e \in E$  if and only if for each edge $e=u v$ there exists a maximum independent set $I$ where $u \in I, v \notin I$ and $u$ is the only neighbor of $v$ in $I$.
\end{lemma}

Lemma \ref{teo2} implies the following result.

\begin{corollary}\label{cor:haynes}
Let $G$ be an $\alpha$-critical graph with no isolated vertices. Then, there exists a maximum independent set $I$ of $G$ such that $v \notin I$ for every vertex $v \in V(G)$.
\end{corollary}

\begin{lemma}\label{thm:iff-universal-cut}
Let $G$ be a graph with $\mathrm{rad}(G) = 1$  and  $\mathrm{diam}(G) = 2$, and let $u$ be a universal cut vertex of $G$. 
Then $G$ is $\chi_\rho$-critical if and only if $G-u$ is a disjoint union of $\alpha$-critical graphs. 
\end{lemma}

\proof	Let $G$ be a graph with $\mathrm{rad}(G) = 1$  and  $\mathrm{diam}(G) = 2$, and let $u$ be a universal cut vertex of $G$. Let further $G$ be a $\chi_\rho$-critical graph. By Proposition \ref{cor1}, $G-u$ is a disjoint union of $\alpha$-critical graphs. 

Conversely, suppose that $G - u$ is a disjoint union of $\alpha$-critical graphs $A_1, \dots, A_r$, where $r \ge 2$. There are two types of edges in $G$: edges with both endpoints in the same connected component of $G-u$, and edges connecting a vertex of some connected component of $G-u$ to $u$.
	
\textbf{Case 1:} Let $e \in E(A_i)$ for some $i \in \{1,\dots,r\}$. Without loss of generality, assume that \( i = 1 \), \ie, \( e \in E(A_1) \).
Since $u$ is a universal vertex in $G-e$, the diameter of $G-e$ remains 2. Additionally, since $A_1$ is $\alpha$-critical, removing $e$ increases $\alpha\left(A_1\right)$. By Lemma \ref{lemma4}, we have:
\begin{align*}
\chi_\rho(G-e)&=|V(G-e)| - \alpha(G-e)  + 1 \\&=|V(G)|-\alpha(A_1-e)-\sum_{j=2}^{r} \alpha(A_j) + 1
\\ &<|V(G)| - \sum_{j=1}^{r} \alpha(A_j) + 1 \\ &=|V(G)| - \alpha(G)  + 1\\ &=\chi_\rho(G).
\end{align*}

Therefore, $\chi_\rho(G-e) < \chi_\rho(G)$ for every edge $e$ in $G-u$.

\textbf{Case 2:} Let \( e = ux \), where \( x \in A_i \)  for some $i \in \{1,\dots,r\}$. In this case, the diameter of \( G - e \) becomes 3, since for any vertex \( y \in A_j \) with \( j \neq i \), the distance between \( x \) and \( y \) is 3 in \( G - e \). Relabel the components so that $x \in A_1$ and $y \in A_2$. We know that $\alpha(G-e) \geq \alpha(G)$.
We first consider the case where $\alpha(G-e)>\alpha(G)$. By Lemma \ref{lemma4}, we have: 
\begin{align*}
\chi_\rho(G-e)&\leq|V(G)|-\alpha(G-e)+1 \\&<|V(G)|-\alpha(G)+1 \\ &=\chi_\rho(G).
\end{align*}
	
We then consider \( \alpha(G - e) = \alpha(G) \). By Corollary \ref{cor:haynes}, there exist maximum independent sets \( I_1 \) in \( A_1 \) and \( I_2 \) in \( A_2 \) such that \( x \notin I_1 \) and \( y \notin I_2 \). $I_1 \cup I_2$ together with maximum inpendent sets from $A_k$ ($k\geq3$) form an \( \alpha(G - e) \)-set $I$ that excludes both \( x \) and \( y \).
Then in graph $G-e$, color $1$ is assigned to every vertex in $I$, whereas color $2$ is assigned to both $x$ and $y$, and distinct colors are assigned to the remaining $|V(G)|-\alpha(G)-2$ vertices. Since $\alpha(G-e)=\alpha(G)$, we obtain a packing coloring of $G-e$ with $|V(G)|-\alpha(G)$ colors. Thus, we have:
\begin{align*}
\chi_\rho(G-e)&\le|V(G)|-\alpha(G) \\&<|V(G)|-\alpha(G)+1 \\ &=\chi_\rho(G). 
\end{align*}
Therefore, $\chi_\rho(G-e) < \chi_\rho(G)$ for every edge $e=ux$ with $x \in G-u$. Since this holds for every edge $e \in E(G)$, $G$ is $\chi_\rho$-critical.
\qed

Lemma \ref{thm:iff-universal-cut} provides the characterization for the case where $G-u$ is disconnected. Accordingly, we now complete the characterization by considering the case where $G-u$ is connected in Lemma \ref{thm:iff-universal}. Our proof will use the next result. 

\begin{lemma}\label{lem:rad3}
If $G$ is an $\alpha$-critical graph with $\operatorname{rad}(G)\ge 3$, then for every two vertices $x$ and $y$ that are at distance $3$ from each other, there exists a maximum independent set that contains neither $x$ nor $y$.
\end{lemma}

\proof 
Let $x \in V(G)$. Since $\operatorname{rad}(G)\ge 3$, we have $\varepsilon_G(x)\ge 3$. Then there exists a vertex $z \in V(G)$ such that $d_G(x,z)\geq3$. Let $x-a-b-y$ be the first four vertices on a shortest $x$--$z$ path. Then $d_G(x,y)=3$. Note that $y=z$ if $\varepsilon_G(x)=3$.
Let $e=ab$. It is well known that for any graph $G$ and edge $e \in E(G)$,
$\alpha(G) \le \alpha(G-e) \le \alpha(G)+1.$ Since $G$ is $\alpha$-critical, $\alpha(G-e)=\alpha(G)+1$.
Let $I$ be a maximum independent set of $G-e$. Then $|I|=\alpha(G)+1$.
We claim that $a,b \in I$. If $a \notin I$, then $I$ is also independent in $G$, since it does not contain both endpoints of $e$, contradiction. Therefore, $a \in I$. By symmetry $b \in I$.
Since $x$ is adjacent to $a$ and $y$ is adjacent to $b$, we have $x,y \notin I$. Now let $I' = I \setminus \{a\}$.
Then $I'$ is an independent set of $G$ with size $\alpha(G)$. Thus, $I'$ is a maximum independent set of $G$ containing neither $x$ nor $y$.
\qed

\begin{lemma}\label{thm:iff-universal}
Let $G$ be a graph with $\mathrm{rad}(G) = 1$  and  $\mathrm{diam}(G) = 2$, and let $u$ be a universal vertex such that $G-u$ is connected. 
Then $G$ is $\chi_\rho$-critical if and only if $G-u$ is $\alpha$-critical and  $\operatorname{rad}(G-u) \geq 3$. 
\end{lemma}

\proof	Let $G$ be a graph with $\mathrm{rad}(G) = 1$  and  $\mathrm{diam}(G) = 2$, and let $u$ be a universal vertex of $G$. First suppose that $G$ is $\chi_\rho$-critical.
	
We first show that $G-u$ is $\alpha$-critical. Let $e \in E(G-u)$. Since $u$ remains universal in $G-e$, we have $\operatorname{diam}(G-e)=2$. Then by Lemma~\ref{lemma4},
$\chi_\rho(G-e)=|V(G)|-\alpha(G-e)+1$ and $\chi_\rho(G)=|V(G)|-\alpha(G)+1$.
Since $G$ is $\chi_\rho$-critical, $\chi_\rho(G-e)<\chi_\rho(G)$ and therefore $\alpha(G-e)>\alpha(G).$
Because $u$ is universal and $\mathrm{diam}(G) = 2$, no maximum independent set of $G$ contains $u$, so $\alpha(G)=\alpha(G-u)$. Similarly, $\alpha(G-e)=\alpha((G-u)-e).$
Thus, $\alpha((G-u)-e)>\alpha(G-u)$ for every edge $e \in E(G-u)$ and hence, $G-u$ is $\alpha$-critical.

Next we show that $\operatorname{rad}(G-u) \geq 3$, \ie, every vertex of $G-u$ has eccentricity at least $3$. Let $x \in V(G-u)$ and suppose to the contrary that $\varepsilon_{G-u}(x)\le 2$. Consider $e=ux$. We claim that $\operatorname{diam}(G-e)=2$.
In $G-e$, any two vertices distinct from $x$ are at distance at most $2$, since $u$ is adjacent to all such vertices. Moreover, for every $y \in V(G-u)$, we have $d_{G-e}(x,y) \le 2$ by assumption. Therefore, $\operatorname{diam}(G-e) \le 2$. Since $u$ and $x$ are nonadjacent in $G-e$, it follows that $\operatorname{diam}(G-e)=2$.
Then by Lemma~\ref{lemma4},
$\chi_\rho(G-e)=|V(G)|-\alpha(G-e)+1.$
Since $G$ is $\chi_\rho$-critical, it follows that $\alpha(G-e)>\alpha(G)$.
Let $I$ be a maximum independent set of $G-e$. If $u \notin I$, then $I$ is also independent in $G$, \ie, $\alpha(G)=\alpha(G-e)$, contradiction. Next assume $u \in I$. Since $u$ is adjacent to every vertex of $G-e$ except $x$, we obtain $I=\{u,x\}$, \ie, $\alpha(G-e)=2$ and so $\alpha(G)=1$, contradicting $\operatorname{diam}(G)=2$.
Therefore, $\operatorname{rad}(G-u) \geq 3$.

Conversely, suppose that $G-u$ is $\alpha$-critical and $\operatorname{rad}(G-u) \geq 3$. We show that $G$ is $\chi_\rho$-critical. There are two types of edges in $G$. Every edge of $G$ is either contained in $G-u$ or incident to $u$. 

\textbf{Case 1:} First let $e \in E(G-u)$. Since $u$ is a universal vertex in $G-e$, we have $\operatorname{diam}(G-e)=2$. Since $G-u$ is $\alpha$-critical,
$\alpha((G-u)-e)>\alpha(G-u).$
As no maximum independent set contains $u$ in $G$ and $G-e$, we have $\alpha(G)=\alpha(G-u)$ and $\alpha(G-e)=\alpha((G-u)-e)$, yielding
$\alpha(G-e)>\alpha(G).$ Then
by Lemma~\ref{lemma4},
$\chi_\rho(G-e)=|V(G)|-\alpha(G-e)+1<|V(G)|-\alpha(G)+1=\chi_\rho(G).$ Therefore, $\chi_\rho(G-e)<\chi_\rho(G)$ for any $e \in E(G-u)$. 

\textbf{Case 2:} For the remaining case let $e=ux$ for some $x \in V(G-u)$. Since $\operatorname{rad}(G-u)\ge 3$, there exists a vertex $y \in V(G-u)$ such that $d_{G-u}(x,y)\ge 3$. Since $x$ has a neighbor that is adjacent to $u$ and $y$ is adjacent to $u$ in $G-e$, we have $d_{G-e}(x,y)=3$. We know that $\alpha(G-e) \geq \alpha(G)$.
We first consider $\alpha(G-e)>\alpha(G)$.
By Lemma~\ref{lemma4}, we have:  
\begin{align*}
		\chi_\rho(G-e)&\leq|V(G)|-\alpha(G-e)+1 \\&<|V(G)|-\alpha(G)+1 \\ &=\chi_\rho(G).
\end{align*}
We then consider the case where \( \alpha(G - e) = \alpha(G) \). Since $G-u$ is $\alpha$-critical and $\operatorname{rad}(G-u)\ge 3$, by Lemma \ref{lem:rad3} there exists a maximum independent set $I$ of $G-u$ such that $x, y \notin I$. Then $I$ is also a maximum independent set of $G-e$ since $\alpha(G-u)=\alpha(G)$. 
Then in graph $G-e$, color $1$ is assigned to every vertex in $I$, color $2$ is assigned to both $x$ and $y$, and distinct colors are assigned to the remaining $|V(G)|-\alpha(G)-2$ vertices. Since $\alpha(G-e)=\alpha(G)$, we obtain a packing coloring of $G-e$ with $|V(G)|-\alpha(G)$ colors. Thus, we have:
	\begin{align*}
		\chi_\rho(G-e)&\le|V(G)|-\alpha(G) \\&<|V(G)|-\alpha(G)+1 \\ &=\chi_\rho(G). 
	\end{align*}
Therefore, $\chi_\rho(G-e) < \chi_\rho(G)$ for every edge $e=ux$ with $x \in G-u$. 

From Cases 1 and 2, it follows that $\chi_\rho(G-e) < \chi_\rho(G)$ for every edge $e \in E(G)$. Therefore, $G$ is $\chi_\rho$-critical.
\qed

Our main result in this section describes the structure of $\chi_\rho$-critical graphs with radius 1, as follows:

\begin{theorem}\label{thm:12}
Let $G$ be a graph with $\mathrm{rad}(G) = 1$  and  $\mathrm{diam}(G) = 2$, and let $u$ be a universal vertex of $G$. 
Then $G$ is $\chi_\rho$-critical if and only if either
\begin{itemize}
\item[(i)] $G-u$ is connected, $\operatorname{rad}(G-u)\ge 3$, and $G-u$ is $\alpha$-critical, or
\item[(ii)] $G-u$ is disconnected and each component of $G-u$ is $\alpha$-critical.
\end{itemize}
\end{theorem}

\begin{proof}
The result follows from Lemmas~\ref{thm:iff-universal-cut} and~\ref{thm:iff-universal}.
\end{proof}

%%%%%%%%%%%%%%%%%%%%%%%%%%%%%%%%%%%%%%%%%%%%%%%%%
\section{$\chi_{\rho}$-critical cactus graphs of radius $2$ and diameter $2$ or $3$}
\label{sec:rad2}
%%%%%%%%%%%%%%%%%%%%%%%%%%%%%%%%%%%%%%%%%%%%%%%%%

In this section, we give a partial characterization of $\chi_\rho$-critical cactus graphs with $\mathrm{rad}(G)=2$. If $\mathrm{rad}(G)=2$, then $\mathrm{diam}(G)\in\{2,3,4\}$. We provide a complete characterization for the cases $\mathrm{diam}(G)=2$ and $\mathrm{diam}(G)=3$, and leave the case $\mathrm{diam}(G)=4$ open.

The next lemma expresses the relationship between the diameter and the radius of a cycle. 

\begin{observation}\label{lemma1}
$\operatorname{diam}(C_n) = \operatorname{rad}(C_n) = \left\lfloor \frac{n}{2} \right\rfloor$ for any cycle \( C_n \).  
\end{observation}

If $G$ is a cactus graph with $\mathrm{rad}(G) = 2$ and it has at least one cycle, then among all cycles \( C \) in \( G \) we select the one with largest cycle and call it as a \emph{main block}, denoted by $C^*$. 
Every other block of \( G \) that is connected to $C^*$ through a cut vertex is called an \emph{outer block}.  

\begin{observation}\label{lem:mainblock}
If $G$ is a cactus graph with $\operatorname{rad}(G)=2$, then every cycle in $G$ is either $C_3$, $C_4$, or $C_5$.
\end{observation}

\proof
By Observation~\ref{lemma1}, $\operatorname{rad}(C_n)=\left\lfloor \frac{n}{2} \right\rfloor$. Thus, if $n\ge 6$, then $\operatorname{rad}(C_n)\ge 3$, contradiction.
\qed

%%%%%%%%%%%%%%%%%%%%%%%%%%%%%%%%%%%%%%%%%%%%%%%%%
\subsection{$\chi_{\rho}$-critical cactus graphs with radius 2 and diameter $2$} \label{sec:rad22}
%%%%%%%%%%%%%%%%%%%%%%%%%%%%%%%%%%%%%%%%%%%%%%%%%

We begin by considering the case where the radius of $G$ and the diameter of $G$ are $2$. We first consider trees, which are acyclic cactus graphs. 

\begin{observation}\label{pro1}
There exists no tree with radius $2$ and diameter $2$. 
\end{observation}

\proof
It is well known that the unique tree with diameter $2$ is a star $K_{1,n}$ for some $n\ge 2$, which has radius $1$. Therefore, no tree has radius $2$ and diameter $2$.
\qed

Next, we consider cactus graphs with radius $2$ and diameter $2$ that contain at least one cycle. Our following result gives a characterization of such graphs.

\begin{lemma}\label{pro2}
Let \( G \) be a cactus graph with $\mathrm{rad}(G) = 2$  and  $\mathrm{diam}(G) = 2$. Then \( G \in \{ C_4, C_5 \} \).
\end{lemma}

\proof Let \( G \) be a cactus graph with $\mathrm{rad}(G) = 2$  and  $\mathrm{diam}(G) = 2$. If $G$ is acyclic, then by Observation \ref{pro1} there exists no such a graph. Next, we suppose that \( G \) contains at least one cycle. Let $C^*$ be the main block. Recall that $C^*$ is either $C_3$, $C_4$, or $C_5$ by Observation \ref{lem:mainblock}.

\textbf{Case 1:} Let $C^* \cong C_3=(x_1,x_2,x_3)$. 

Since $C^*$ has diameter $1$, $G$ is not isomorphic to $C^*$. Therefore, there exists a vertex $x \in V(G)\setminus V(C^*)$.
We first show that $x$ has a neighbor in $C^*$. If not, since $G$ is connected, there is a shortest path from $x$ to some $x_i \in V(C^*)$ of length at least $2$, say $x_1$.
Then $d_G(x,x_j)\ge 3$ for any $j \in \{2,3\}$, contradiction.
Thus, every vertex in $V(G)\setminus V(C^*)$ is adjacent to a vertex of $C^*$. 
If two such vertices are adjacent to different vertices of $C^*$, say $u$ to $x_1$ and $v$ to $x_2$, then $u-x_1-x_2-v$ is a path of length $3$, contradicting with $\operatorname{diam}(G)=2$. 
Therefore, all vertices of $V(G)\setminus V(C^*)$ are adjacent to a single vertex of $C^*$, say $x_1$.
It follows that $x_1$ is a universal vertex of $G$, \ie, $\varepsilon_G(x_1)=1$, which implies $\operatorname{rad}(G)=1$, contradiction. Therefore, no cactus graph with radius $2$ and diameter $2$ contains a cycle isomorphic to $C_3$.

\textbf{Case 2:} Let \( C^* \cong C_4=(x_1, x_2, x_3, x_4) \).

By Observation~\ref{lemma1}, $\operatorname{diam}(C^*)=\operatorname{rad}(C^*)=2$. 
We claim that $V(G)=V(C^*)$. Suppose to the contrary that there exists a vertex $x \in V(G)\setminus V(C^*)$. Since $G$ is connected,  there is a shortest path from $x$ to some $x_i \in V(C^*)$ of length at least $1$, say $x_1$.  Then $d_G(x,x_3)\ge 3$, contradiction.  Therefore, $G \cong C_4$.
	
\textbf{Case 3:} Let \( C^* \cong C_5=(x_1, x_2, x_3, x_4, x_5) \).
	
By Observation~\ref{lemma1}, $\operatorname{diam}(C^*)=\operatorname{rad}(C^*)=2$. 
We claim that $V(G)=V(C^*)$. Suppose to the contrary that there exists a vertex $x \in V(G)\setminus V(C^*)$. Since $G$ is connected,  there is a shortest path from $x$ to some $x_i \in V(C^*)$ of length at least $1$, say $x_1$.  Then $d_G(x,x_j)\ge 3$ for any $j \in \{3,4\}$, contradiction.  Therefore, $G \cong C_4$.
	
As a result, \( G \) is either \( C_4 \) or \( C_5 \).
\qed

Next, we express our main result about $\chi_{\rho}$-critical cactus graphs with radius $2$ and diameter $2$.

\begin{theorem}\label{teo3}
If $G$ is a cactus graph with $\mathrm{rad}(G) = 2$  and  $\mathrm{diam}(G) = 2$, then $G$ is $\chi_{\rho}$-critical if and only if $G \cong C_5$.
\end{theorem}

\proof
Let $G$ be a $\chi_\rho$-critical cactus graph with $\mathrm{rad}(G)=2$ and $\mathrm{diam}(G)=2$. By Lemma~\ref{pro2}, $G \in \{C_4,C_5\}$. Since $\chi_\rho(C_4)=\chi_\rho(P_4)=3$, it is clear that $C_4$ is not $\chi_\rho$-critical. However, $C_5$ is $\chi_\rho$-critical since $\chi_\rho(C_5)=4$ and $\chi_\rho(P_5)=3$.
Therefore, $G \cong C_5$.
\qed 

%%%%%%%%%%%%%%%%%%%%%%%%%%%%%%%%%%%%%%%%%%%%%%%%%
\subsection{$\chi_{\rho}$-critical cactus graphs with radius 2 and diameter $3$}
\label{sec:rad23}
%%%%%%%%%%%%%%%%%%%%%%%%%%%%%%%%%%%%%%%%%%%%%%%%%

Next, we consider \( \chi_{\rho} \)-critical cactus graphs with radius $2$ and diameter $3$.

We first consider trees, which are acyclic cactus graphs. Lemma \ref{pro3} provides a characterization of \( \chi_{\rho} \)-critical trees with radius \( 2 \) and diameter \( 3 \).

\begin{lemma}~{\rm \cite{bresar-2022}}\label{pro3} 
Let \( G \) be a tree with $\mathrm{rad}(G) = 2$  and  $\mathrm{diam}(G) = 3$. Then \( G \) is \( \chi_{\rho} \)-critical if and only if \( G \cong P_4 \).
\end{lemma}

Next, we consider cactus graphs with at least one cycle. Let $C^*$ be the main block. Recall that $C^*$ is either $C_3$, $C_4$, or $C_5$ by Observation \ref{lem:mainblock}. We now introduce a family of cactus graphs that will be used in the subsequent arguments. Two example graphs from this family are shown in Figures \ref{fig4} and \ref{fig5}.

Let ${\mathcal{G}_{q}^{(r)}(k_{1},m_{1};k_{2},m_{2};\dots;k_{q},m_{q})}$ be a family of cactus graphs \( G \) with main block isomorphic to $C_r= (x_1, x_2, \dots, x_r)$ such that \( G \) has exactly $q$ consecutive cut vertices $ x_1, x_2, \dots, x_q \in V(C_r)$, where $1 \leq q \leq r$. Each cut vertex $x_i$ is incident to outer blocks that are either $K_2$ or $C_3 = (x_i, u_{ij}, v_{ij})$, where $1 \leq i \leq q$ and $1 \leq j \leq m_i$. In addition, $k_i$ and $m_i$ are the number of $K_2$ and $C_3$ blocks incident to $x_i$, respectively.
Here \( k_i, m_i \in \mathbb{N}_0 \) and \( k_i + m_i \ge 1 \) for every $1 \leq i \leq q$.

\begin{figure}[H]
\centering
\begin{minipage}{0.47\textwidth}
\centering
\begin{tikzpicture}
	\tikzset{unode/.style = {circle, draw=cyan!30!black, thick, fill=white, inner sep=2.3pt, minimum size=2.3pt}}
	\tikzset{uedge/.style = {draw=cyan!20!black, thick}}
	
	\begin{scope}
		
		% ---- center block C5 ----
		\foreach \x [count=\i] in {18,90,162,234,306}{
			\coordinate (o\i) at (\x:1.8cm) {};
		}
		
		% ---- outer triangle vertices (i1 removed) ----
		\foreach \x [count=\i] in {145,160,175,195,210}{
			\coordinate (i\i) at (\x:2.6cm) {};
		}
		
		% ---- edges ----
		\path[uedge] (o1)--(o2)--(o3)--(o4)--(o5)--(o1);
		\path[uedge] (o3)--(i1);
		\path[uedge] (o3)--(i2)--(i3)--(o3);
		\path[uedge] (o3)--(i4)--(i5)--(o3);
		
		% ---- draw nodes ----
		\foreach \x [count=\i] in {18,90,162,234,306}{
			\node[unode] (vx\i) at (\x:1.8cm) {};
		}
		
		\foreach \x [count=\i] in {145,160,175,195,210}{
			\node[unode] (vi\i) at (\x:2.6cm) {};
		}
		
		% ---- labels for C5 ----
		\node[below right] at (o3.west) {\scriptsize $x_1$};
		\node[above right] at (o2.north east) {\scriptsize $x_5$};
		\node[below right]  at (o1.north west) {\scriptsize $x_4$};
		\node[below right]  at (o4.south west) {\scriptsize $x_2$};
		\node[below right] at (o5.south east) {\scriptsize $x_3$};
		
		% ---- labels for outer triangles ----
		\node[above left]  at (i2.north west) {\scriptsize $u_{11}$};
		\node[below left]  at (i3.south west) {\scriptsize $v_{11}$};
		\node[above left]  at (i4.south west) {\scriptsize $u_{12}$};
		\node[below left] at (i5.south east) {\scriptsize $v_{12}$};
		
	\end{scope}
\end{tikzpicture}
\caption{$\mathcal{G}_{1}^{(5)}(1,2).$}
\label{fig4}
\end{minipage}
\hfill
\begin{minipage}{0.47\textwidth}
\centering
\begin{tikzpicture}[scale=0.9]
	\tikzset{unode/.style = {circle, draw=cyan!30!black, thick, fill=white, inner sep=2pt, minimum size=2pt}}
	\tikzset{uedge/.style = {draw=cyan!20!black, thick}}
	
	\begin{scope}
		
		% ---- center block C4 vertices ----
		\foreach \x [count=\i] in {45,135,225,315}{
			\coordinate (o\i) at (\x:1.4cm) {};
		}
		
		% ---- outer vertices (reduced) ----
		\foreach \x [count=\i] in {135,150,165,185,200,215,230,245,260}{
			\coordinate (i\i) at (\x:2.3cm) {};
		}
		
		% ---- edges of C4 ----
		\path[uedge] (o1)--(o2)--(o3)--(o4)--(o1);
		
		% attached at x1 = o2
		\path[uedge] (o2)--(i1);
		\path[uedge] (o2)--(i2)--(i3)--(o2);
		\path[uedge] (o2)--(i4)--(i5)--(o2);
		
		% attached at x2 = o3
		\path[uedge] (o3)--(i6);
		\path[uedge] (o3)--(i7);
		\path[uedge] (o3)--(i8)--(i9)--(o3);
		
		% ---- draw nodes ----
		\foreach \x in {45,135,225,315}{
			\node[unode] at (\x:1.4cm){};
		}
		\foreach \x in {135,150,165,185,200,215,230,245,260}{
			\node[unode] at (\x:2.3cm){};
		}
		
		% ---- labels for C4 ----
		\node at ($(o2)+(-0.006,0.3)$) {\scriptsize $x_1$};
		\node at ($(o3)+(-0.3,0.1)$) {\scriptsize $x_2$};
		\node[below right] at (o4.south east) {\scriptsize $x_3$};
		\node[above right] at (o1.north east) {\scriptsize $x_4$};
		
		% triangles at x1
		\node[above left]  at (i2.north west) {\scriptsize $u_{11}$};
		\node[below left]  at (i3.south west) {\scriptsize $v_{11}$};
		\node[above left]  at (i4.south west) {\scriptsize $u_{12}$};
		\node[below left]  at (i5.south east) {\scriptsize $v_{12}$};
		
		% triangle at x2 (only one left)
		\node[below] at (i8.north east) {\scriptsize $u_{21}$};
		\node[below] at (i9.east)      {\scriptsize $v_{21}$};
		
	\end{scope}
\end{tikzpicture}
\caption{$\mathcal{G}_{2}^{(4)}(1,2;2,1).$}
\label{fig5}
\end{minipage}
\end{figure}

%%%%%%%%%%%%%%%%%%%%%%%%%%%%%%%%%%%%%%%%%%%%%%%%%
\subsubsection{$\chi_{\rho}$-critical cactus graphs with radius 2 and diameter $3$, where $C_5$ is the main block}
%%%%%%%%%%%%%%%%%%%%%%%%%%%%%%%%%%%%%%%%%%%%%%%%%

We now proceed with cactus graphs with radius 2 and diameter $3$ with $C_5$ as their main block. The next result provides the packing chromatic number of a graph $G \in \mathcal{G}_{1}^{(5)}(k_1,m_1)$.

\begin{proposition}\label{pro4} 
Let $G \in \mathcal{G}_{1}^{(5)}(k_1,m_1)$ and $k_1, m_1 \geq 0$. Then
\[
\chi_\rho(G) =
\begin{cases}
4, & \text{if } m_1 = 0, \\[6pt]
m_1+3, & \text{if } m_1 \geq 1.
\end{cases}
\]
\end{proposition}

\proof Let $G \in \mathcal{G}_{1}^{(5)}(k_1,m_1)$. First let $m_1 = 0$. Then $k_1 \geq 1$ by the definition of $\mathcal{G}_{1}^{(5)}(k_1,m_1)$.  It is well known that for every subgraph \( H \subseteq G \), we have \( \chi_\rho(H) \leq \chi_\rho(G) \). Therefore, since \( \chi_\rho(C_5) = 4 \), it follows that \( \chi_\rho(G) \geq 4 \). To show that \( \chi_\rho(G) = 4 \), it suffices to construct a feasible $4$-packing coloring of \( G \). Let $x_1$ be the cut vertex of $G$. Define a vertex coloring $c:V(G)\to[4]$ as follows: $c(x_2) = c(x_4)=1$, $c(x_3)=2$, $c(x_1) = 3$, $c(x_5) = 4$, and assign color $1$ to each leaf attached to $x_1$. Therefore,  $\chi_\rho(G) =  4$. 
	
Now, let \( m_1 \geq 1 \). Obviously $|V(G)|=2m_1+k_1+5$. Note that any packing coloring of $G$ of size $\ell$ partitions $V(G)$ into $\ell$ color classes $X_1, X_2, \dots,X_{\ell}$ such that each $X_i$ is an $i$-packing for each $i \in \{1, 2, \dots, \ell\}$. So, $X_1$ contains at most $\alpha(G)=m_1+k_1+2$ vertices, $X_2$ contains at most $2$ vertices since there are no $3$ vertices whose pairwise distances are equal to $3$, and each $X_j$ contains exactly one vertex for each $j \in \{3, 4, \dots, \ell\}$ since $G$ has diameter $3$. Thus, we obtain the following inequality:
$$
|V(G)|=|X_1| + \dots + |X_{\ell}| \leq \alpha(G) + 2 + (\ell - 2) = \alpha(G) + \ell.
$$
By rearranging the terms, it follows that
\[
\ell \geq |V(G)| - \alpha(G)
= (2m_1+k_1+5) - (m_1+k_1+2)
= m_1+3.
\]
So every packing coloring of $G$ has at least $m_1 + 3$ colors,
hence, $\chi_{\rho}(G) \geq m_1 + 3.$
For the upper bound, we construct a packing coloring of $G$ with $m_1+3$ colors.  
Let $x_1$ be the cut vertex of $G$. Since $m_1 \geq 1$, there exist $m_1$ number of $C_3$ blocks $(x_1,u_{1i},v_{1i})$, where $1 \leq i \leq m_1$. Define a packing coloring $c:V(G)\to[m_1+3]$ as follows: $c(x_2) = c(x_4)=1$, $c(x_3)=c(v_{11})=2$, $c(x_1) = 3$, $c(x_5) = 4$.
Moreover, assign color $1$ to all leaves attached to $x_1$, and all vertices $u_{1i}$.
Lastly, assign distinct colors $5,6,\dots,m_1+3$ to the remaining $m_1-1$ vertices $v_{1i}$ for $i\geq 2$ (see Figure \ref{fig6}). Thus, we conclude that  $\chi_\rho(G) =  m_1+3$.
\qed

\begin{figure}[H]
\centering		
\begin{tikzpicture}
\tikzset{unode/.style = {
circle, 
draw=cyan!30!black, 
thick,
fill=white,
inner sep=3.2pt,
minimum size=7pt } }
\tikzset{uedge/.style = {
draw=cyan!20!black, 
thick} }
\tikzset{knode/.style = {
circle, 
draw=cyan!30!black, 
thin,
fill=myblack,
inner sep=0.5pt,
minimum size=0.5pt } }

\begin{scope}[xshift=5cm]
\foreach \x [count=\i] in {18,90,162,234,306}{
\coordinate (o\i) at (\x:1.8cm) {};
}
\foreach \x [count=\i] in {125, 140, 155, 170, 185, 200}{
\coordinate (i\i) at (\x:3cm) {};
}
			
\path[uedge] (o1) -- (o2) -- (o3) -- (o4)-- (o5)--(o1);
\path[uedge] (o3) -- (i1);
\path[uedge] (o3) -- (i2);
\path[uedge] (o3) -- (i3) -- (i4) -- (o3);
\path[uedge] (o3) -- (i5) -- (i6) -- (o3);
			
\foreach \x [count=\i] in {18,90,162,234,306}{
\node[unode] at (\x:1.8cm){};
}
\foreach \x [count=\i] in {125, 140, 155, 170, 185, 200}{
\node[unode] at (\x:3cm){};
}

\node at (o1) {\scriptsize $1$};
\node at (o2) {\scriptsize $4$};
\node at (o3) {\scriptsize $3$};
\node at (o4) {\scriptsize $1$};
\node at (o5) {\scriptsize $2$};
\node at (i1) {\scriptsize $1$};
\node at (i2) {\scriptsize $1$};
\node at (i3) {\scriptsize $1$};
\node at (i4) {\scriptsize $2$};
\node at (i5) {\scriptsize $1$};
\node at (i6) {\scriptsize $5$};

\node[right] at ($(o3.west)+(0.15,0)$) {\scriptsize $x_{1}$};
\end{scope}
		
\end{tikzpicture}
\caption{$\mathcal{G}_{1}^{(5)}(2,2)$-An example of a $5$-packing coloring.}
\label{fig6}
\end{figure}
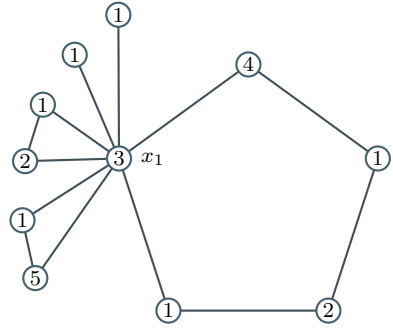

\begin{proposition}\label{pro5} 
Let $G \in \mathcal{G}_{1}^{(5)}(k_1,m_1)$. If $k_1 > 0$, then $G$ is not $\chi_{\rho}$-critical. 
\end{proposition}

\proof Suppose that $G \in \mathcal{G}_{1}^{(5)}(k_1,m_1)$  with $k_1 > 0$ is a \( \chi_\rho \)-critical graph, \ie, $G$ contains at least one leaf \( v \). Let \( e = x_1v \) be the edge incident to \( v \), where \( x_1 \) is the cut vertex of $G$.  Since \( v \) is a leaf, the removal of \( e \) splits the graph $G-e$ into two components \( H_1=G[V(G) \setminus \{v\}] \) and $H_2$, which is an isolated vertex $v$. Therefore, we have $\chi_\rho(G-e)=\max \left(\chi_\rho\left(H_1\right), \chi_\rho\left(H_2\right)\right)=\chi_\rho(H_1)$ since $\chi_\rho(H_2)=1$. For $k_1 = 1$ and $m_1=0$, we have $H_1 \cong C_5$ and $\chi_\rho(G-e)=\chi_\rho(C_5)=4$. In addition, by Proposition~\ref{pro4}, $\chi_\rho(G)=4$, which contradicts the assumption that \( G \) is \( \chi_\rho \)-critical. In the other cases, $H_1 \in \mathcal{G}_{1}^{(5)}(k_1-1,m_1)$ and its number of $C_3$ blocks ($m_1$) is equal to that of $G$. Then $\chi_\rho(G - e) = \chi_\rho(G) =m_1+3$ by Proposition~\ref{pro4}. Therefore, $G$ is not $\chi_{\rho}$-critical.
\qed

\begin{proposition}\label{pro6}
Let $G \cong \mathcal{G}_{1}^{(5)}(0,1)$. Then $G$ is not $\chi_{\rho}$-critical. 
\end{proposition}

\proof Let $G \in \mathcal{G}_{1}^{(5)}(0,1)$,	and let the single $C_3$ block attached to $x_1$ be $(x_1,u_{11},v_{11})$. Consider the edge $e=u_{11}v_{11}$. Then we have $G-e \cong \mathcal{G}_{1}^{(5)}(2,0)$. By Proposition~\ref{pro4}, we have $\chi_{\rho}(G) = \chi_{\rho}(G - e) = 4$. Therefore, $G$ is not $\chi_{\rho}$-critical.
\qed

\begin{lemma}\label{pro7}
Let $G \in \mathcal{G}_{1}^{(5)}(k_1,m_1)$. Then $G$ is $\chi_{\rho}$-critical if and only if $k_1=0$ and $m_1\geq 2$.
\end{lemma}

\proof Let $G \in \mathcal{G}_{1}^{(5)}(k_1,m_1)$ be a $\chi_{\rho}$-critical graph. Then by Propositions \ref{pro5} and \ref{pro6}, we have $k_1=0$ and $m_1\geq 2$.
	
For the converse direction, let $G \in \mathcal{G}_{1}^{(5)}(k_1,m_1)$ with $k_1=0$ and $m_1\geq 2$. We know that  $\chi_{\rho}(G)=m_1+3$ by Proposition \ref{pro4}.  We show that for every edge $e$ of $G$, we have $\chi_{\rho}(G-e) < \chi_{\rho}(G)=m_1+3$. Note that the edges of \(G\) can be classified into five distinct types according to their position in the structure of \(G\). We proceed by a case analysis according to these five types.
	
\textbf{Case 1:} Consider the edge \( e = x_1x_2 \) 
(which is symmetric to the edge \( e = x_1x_5 \)). In the graph \( G-e \), define a packing coloring \( c : V(G-e) \to [m_1+2] \) as follows:
$c(x_2) = c(x_5) = 1, 
c(x_4) = c(v_{11}) = 2,
c(x_3) = c(v_{12}) = 3,
c(x_1) = 4.$
Assign color \(1\) to all vertices \(u_{1i}\) for \(1 \le i \le m_1\). 
Finally, assign distinct colors \(5, 6, \dots, m_1+2\) to the remaining \(m_1-2\) vertices \(v_{1i}\) with \(i \ge 3\). Then \(\chi_\rho(G-e) \le m_1+2\).
	
\textbf{Case 2:} Consider the edge $e=x_2x_3$ 
(which is symmetric to the edge $e=x_4x_5$). In the graph \( G-e \), define a packing coloring \( c : V(G-e) \to [m_1+2]\) as follows:
$c(x_2)=c(x_3)=c(x_5)= 1, 
c(x_4)=c(v_{11}) = 2,
c(x_1) = 3.$
Assign color \(1\) to all vertices \(u_{1i}\) for \(1 \le i \le m_1\). 
Finally, assign distinct colors \(4, 5, \dots, m_1+2\) to the remaining \(m_1-1\) vertices \(v_{1i}\) with \(i \ge 2\). 
Therefore, \(\chi_\rho(G-e) \le m_1+2\).
	
\textbf{Case 3:} Consider the edge $e=x_3x_4$. In the graph \( G-e \), define a packing coloring \( c : V(G-e) \to [m_1+2] \) as follows:
$c(x_2)=c(x_5)= 1, 
c(x_3)=c(x_4)=c(v_{11})= 2,
c(x_1) = 3.$
Assign color \(1\) to all vertices \(u_{1i}\) for \(1 \le i \le m_1\). 
Finally, assign distinct colors \(4, 5, \dots, m_1+2\) to the remaining \(m_1-1\) vertices \(v_{1i}\) with \(i \ge 2\). 
Therefore, \(\chi_\rho(G-e) \le m_1+2\).
	
\textbf{Case 4:} Consider the edge $e=x_1v_{11}$ 
(which is symmetric to the edge $e = x_1u_{1j}$ for all $1 \leq j \leq m_1$ and also to the edge $e = x_1v_{1j}$ for all $2 \leq j \leq m_1$). In the graph \( G-e \), define a packing coloring \( c : V(G-e) \to [m_1+2] \) as follows:
$c(x_2)=c(x_5)= 1, 
c(x_3)=c(v_{11})=c(v_{12})= 2,
c(x_1)=3,
c(x_4)=4.$
Assign color \(1\) to all vertices \(u_{1i}\) for \(1 \le i \le m_1\). 
Finally, assign distinct colors \(5, 6, \dots, m_1+2\) to the remaining \(m_1-2\) vertices \(v_{1i}\) with \(i \ge 3\). 
Therefore, \(\chi_\rho(G-e) \le m_1+2\).
	
\textbf{Case 5:} Consider the edge $e=u_{11}v_{11}$ (which is symmetric to the edge $e = u_{1j}v_{1j}$ for all $2 \leq j \leq m_1$).  Clearly, $G-e$ is isomorphic to the graph $\mathcal{G}_{1}^{(5)}(2,m_1-1)$ and $\chi_{\rho}(\mathcal{G}_{1}^{(5)}(2,m_1-1)) = m_1+2$ for $m_1\geq2$ by Proposition \ref{pro4}. We conclude that  $\chi_\rho(G-e) =  m_1+2$.
	
Since these cases cover all possible edges of $G$, the proof is complete.
\qed

Next, we present several results concerning cactus graphs with radius $2$ and diameter $3$ whose main block is $C_5$ and which have two cut vertices.

\begin{proposition}\label{pro8}
Let $G \in \mathcal{G}_{2}^{(5)}(k_{1},m_{1};k_{2},m_{2})$ and $k_1, k_2 \geq 0, \; m_1, m_2 \geq 0$. Let further $t = m_1+m_2$. Then
\[
\chi_\rho(G) =
\begin{cases}
4, & \text{if } m_1=m_2= 0, \\[6pt]
t+3, & \text{if either } m_1 = 0 \text{ or } m_2 = 0, \text{ and } t \geq 1, \\[6pt]
t+2, & \text{if }  m_1, m_2 \geq 1.
\end{cases}
\]
\end{proposition}

\proof Let $G \in \mathcal{G}_{2}^{(5)}(k_{1},m_{1};k_{2},m_{2})$ and $t = m_1+m_2$. We first consider the case $m_1=m_2 = 0$. Then $k_1, k_2 \geq 1$. Since $C_5 \subseteq G$, it follows that  $\chi_\rho(G) \geq \chi_\rho(C_5) = 4$. We construct a feasible $4$-packing coloring of $G$.  Let $x_1$ and $x_2$ be cut vertices of $G$. Define a packing coloring $c:V(G)\to[4]$ as follows: $c(x_3)=c(x_5)=1, c(x_4)=2, c(x_1)=3, c(x_2)=4$, and assign color $1$ to every leaf attached to either $x_1$ or $x_2$.
Therefore,  $\chi_\rho(G) =  4$.
	
Next, we consider the case where either $m_1 = 0$ or $m_2 = 0$, and $t=m_1+m_2\geq 1$. By symmetry, we may assume without loss of generality that $m_2=0$. Then $m_1 = t \geq 1$. In this case $G$ has a subgraph $\mathcal{G}_{2}^{(5)}(k_{1},t;0,0) \in \mathcal{G}_{1}^{(5)}(k_1, t)$. Hence, Proposition~\ref{pro4} implies that \( \chi_\rho(G) \ge t + 3 \). 
To prove equality, it suffices to construct a feasible packing coloring of \(G\) with $t + 3$ colors.   Let $x_1$ and $x_2$ be cut vertices of $G$.  Since $t \geq 1$, attach $t$ number of $C_3$ blocks, namely cycles $(x_1,u_{1i},v_{1i})$ to $x_1$, where \(1 \le i \le t\). Define a packing coloring $c:V(G)\to [t+3]$ as follows: $c(x_3)=c(x_5)=1, c(x_4)=c(v_{11})=2, c(x_1)=3, c(x_2)=4.$ 
Moreover, assign color $1$ to all leaves attached to $x_1$, $x_2$, and all vertices $u_{1i}$.
Finally, assign distinct colors $5,6,\dots,t+3$ to the remaining $t-1$ vertices $v_{1i}$ for \(i \ge 2\). Clearly, $c$ is a packing coloring using $t+3$ colors, and we conclude that  $\chi_\rho(G) =  t+3$.
	
In the last case where $m_1, m_2 \geq 1$, obviously $|V(G)|=2t+k_1+k_2+5$. Note that any packing coloring of $G$ of size $\ell$ partitions $V(G)$ into $\ell$ color classes $X_1, X_2, \dots, X_{\ell}$ such that each $X_i$ is an $i$-packing for each $i \in \{1, 2, \dots, \ell\}$. So, $X_1$ contains at most $\alpha(G)=t+k_1+k_2+2$ vertices, $X_2$ contains at most $3$ vertices since there are no $4$ vertices whose pairwise distances are equal to $3$, and each $X_j$ contains exactly one vertex for each $j \in \{3, 4, \dots, \ell\}$ since $G$ has diameter $3$. Thus, we obtain the following inequality:
$$|V(G)| \leq \alpha(G) + 3 + (\ell - 2) = \alpha(G) + \ell + 1.$$
By rearranging the terms, it follows that
\[
\ell \geq |V(G)| - \alpha(G) - 1
= (2t+k_1+k_2+5) - (t+k_1+k_2+2) - 1
= t+2.
\]
	Hence, we conclude that $\chi_{\rho}(G) \geq t+2.$
For the upper bound, we construct a packing coloring of $G$ with $t+2$ colors. Let $x_1$ and $x_2$ be cut vertices of $G$. Since $ m_1, m_2 \geq 1$, we consider the case where $m_1$ number of $C_3$ blocks 
$(x_1,u_{1i},v_{1i})$ for $1 \leq i \leq m_1$, and $m_2$ number of $C_3$ blocks $(x_2,u_{2j},v_{2j})$ for $1 \leq j \leq m_2$ exist. Define a packing coloring $c:V(G)\to[t+2]$ as follows: $c(x_3)=c(x_5)=1, c(x_4)=c(v_{11})=c(v_{21})=2, c(x_1)=3, c(x_2)=4.$ Moreover, assign color $1$ to all leaves attached to $x_1$, $x_2$, and all vertices $u_{1i}, u_{2j}$.
Lastly, assign distinct colors $5,6,\dots,t+2$ to the remaining $t-2$ vertices $v_{1i}$ and $v_{2j}$. Clearly, $c$ is a packing coloring using $t+2$ colors, and we conclude that  $\chi_\rho(G) =  t+2$. 
\qed

We now define a new graph family $\mathcal{H}$, which will be used in the subsequent proofs.
Let $\mathcal{H}(k_1, m_1; k_2, m_2)$ be the family of graphs obtained from two adjacent vertices $u_1$ and $u_2$ by attaching to each $u_i$ exactly $k_i$ copies of $K_2$ and $m_i$ copies of $C_3$, for $i=1,2$.
Here \( k_i, m_i \in \mathbb{N}_0 \) and \( k_i + m_i \ge 1 \) for each $i=1,2$. An example member of $\mathcal{H}$ is shown in Figure \ref{fig7}.

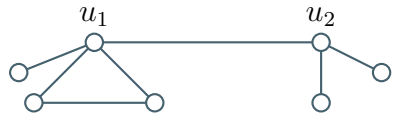
\begin{figure}[H]
\centering
\begin{tikzpicture}
\tikzset{unode/.style = {
circle, draw=cyan!30!black, thick, fill=white, inner sep=2.3pt, minimum size=2.3pt } }
\tikzset{uedge/.style = {
draw=cyan!30!black, 
thick} }
	
\coordinate (u1) at (0,0);
\coordinate (a1) at (-0.8,-0.8);
\coordinate (a2) at (0.8,-0.8);
\coordinate (a3) at (-1,-0.4); 
		
\coordinate (u2) at (3,0);
\coordinate (b1) at (3,-0.8);
\coordinate (b2) at (3.8,-0.4);
		
\path[uedge] (u1) -- (u2);      
\path[uedge] (u1) -- (a1);
\path[uedge] (u1) -- (a2);
\path[uedge] (a1) -- (a2);       
\path[uedge] (u1) -- (a3);
\path[uedge] (u2) -- (b1);
\path[uedge] (u2) -- (b2);
		
\foreach \x in {u1,u2,a1,a2,a3,b1,b2}{
\node[unode] at (\x) {};
}
		
\node[above=2pt] at (u1) {$u_1$};
\node[above=2pt] at (u2) {$u_2$};
\end{tikzpicture}
\caption{$\mathcal{H}(1, 1; 2, 0)$.}
\label{fig7}
\end{figure}

\begin{lemma}\label{pro9}
Let $G \in \mathcal{G}_{2}^{(5)}(k_{1},m_{1};k_{2},m_{2})$ with $k_1, k_2 \geq 0, \; m_1, m_2 \geq 0$. Then $G$ is not $\chi_{\rho}$-critical.
\end{lemma}

\proof
Let $G \in \mathcal{G}_{2}^{(5)}(k_{1},m_{1};k_{2},m_{2})$ and $t = m_1+m_2$. We first consider $t =0$. Then by the definition of $\mathcal{G}_{2}^{(5)}(k_{1},m_{1};k_{2},m_{2})$, $k_1 \geq 1$ and $k_2 \geq 1$. Hence, $G$ has at least two $K_2$ blocks. Consider an edge $e$ contained in a $K_2$ block. Since $C_5$ is a subgraph of $G-e$, it follows that $\chi_{\rho}(G-e) \geq 4$. By Proposition \ref{pro8}, we also have $\chi_{\rho}(G)=4$. Therefore, $\chi_\rho(G-e)=\chi_\rho(G)$ and $G$ is not $\chi_\rho$-critical.
	
We now assume that exactly one of $m_1$ or $m_2$, say $m_1$, is nonzero. Then $t=m_1 \geq 1$ and by the definition of $\mathcal{G}_{2}^{(5)}(k_{1},m_{1};k_{2},m_{2})$, $k_2 \geq 1$, \ie, there is at least one pendant edge at $x_2$. Consider such an edge $e$. If $k_2=1$,
then $G-e$ has a subgraph $\mathcal{G}_{2}^{(5)}(k_{1},t;0,0) \in \mathcal{G}_{1}^{(5)}(k_1, t)$. Hence, Proposition~\ref{pro4} implies that \( \chi_\rho(G-e) = t + 3 \). 
On the other hand, Proposition~\ref{pro8} yields $\chi_\rho(G)=t+3$.
Therefore, $G$ is not $\chi_\rho$-critical.
If $k_2\ge 2$, then $G-e$ has a subgraph $\mathcal{G}_{2}^{(5)}(k_{1},t;k_2-1,0)$. Therefore, Proposition~\ref{pro8} implies that \( \chi_\rho(G-e)=\chi_\rho(G) = t + 3 \).
Thus, $G$ is not $\chi_\rho$-critical.
	
In the final case, suppose that $m_1, m_2 \geq 1$. Consider an edge $e=x_1x_5$ of \(G\), where $x_1$ is a cut vertex of $G$. 
Then $G-e$ contains the graph $H \cong \mathcal{H}(k_1, m_1; k_2, m_2)$ as an induced subgraph. 
We show that \( \chi_\rho(H) \ge t + 2 \) with $m_1, m_2 \geq 1$. Obviously $|V(H)|=2t + k_1 + k_2 + 2$. Note that any packing coloring of $H$ of size $\ell$ partitions $V(H)$ into $\ell$ color classes $X_1, X_2, \dots, X_{\ell}$ such that each $X_i$ is an $i$-packing for each $i \in \{1, 2, \dots, \ell\}$. So, $X_1$ contains at most $\alpha(H)=t+k_1+k_2$ vertices, $X_2$ contains at most $2$ vertices since there are no $3$ vertices whose pairwise distances are equal to $3$, and each $X_j$ contains exactly one vertex for each $j \in \{3, 4, \dots, \ell\}$ since $H$ has diameter $3$. Thus, we obtain the following inequality:
$$|V(H)| \leq \alpha(H) + 2 + (\ell - 2) = \alpha(H) + \ell.$$
By rearranging the terms, it follows that
\[
\ell \geq |V(H)| - \alpha(H) 
= (2t + k_1 + k_2 + 2) - (t+k_1+k_2) 
= t+2.
\]
Hence, we conclude that $\chi_{\rho}(H) \geq t+2.$
Since $H$ is a subgraph of $G-e$, it follows that \( \chi_\rho(G - e) \ge \chi_\rho(H) \ge t + 2 \). On the other hand, when \( m_1, m_2 \ge 1 \), Proposition~\ref{pro8} yields \( \chi_\rho(G) = t + 2 \). 
Thus, \( \chi_\rho(G - e) = \chi_\rho(G)=t+2 \), and consequently \( G \) is not \( \chi_\rho \)-critical. 
\qed

%%%%%%%%%%%%%%%%%%%%%%%%%%%%%%%%%%%%%%%%%%%%%%%%%
\subsubsection{$\chi_{\rho}$-critical cactus graphs with radius 2 and diameter $3$, where $C_4$ is the main block}
%%%%%%%%%%%%%%%%%%%%%%%%%%%%%%%%%%%%%%%%%%%%%%%%%
In this section, we will consider cactus graphs with radius 2 and diameter $3$ with $C_4$ as their main block. We begin by recalling the definition of the friendship graph, which will be used in our proofs.

\textit{A friendship graph} $T_n$, shown in Figure \ref{fig8}, consists of $n$ triangles with exactly one common vertex, where $n \geq 1$.

\begin{figure}[H]
\centering
\begin{tikzpicture}
\tikzset{unode/.style = {
circle, 
draw=cyan!30!black, 
thick,
fill=white,
inner sep=2.3pt,
minimum size=2.3pt } }
\tikzset{uedge/.style = {
draw=cyan!20!black, 
thick} }
		
\coordinate (c) at (0,0);
\coordinate (a1) at (90:1.5cm);
\coordinate (a2) at (30:1.5cm);
\coordinate (a3) at (210:1.5cm);
\coordinate (a4) at (150:1.5cm);
\coordinate (a5) at (-30:1.5cm);
\coordinate (a6) at (-90:1.5cm);
		
\path[uedge] (c) -- (a1) -- (a2) -- cycle;
\path[uedge] (c) -- (a3) -- (a4) -- cycle;
\path[uedge] (c) -- (a5) -- (a6) -- cycle;
		
\node[unode] at (c) {};
\foreach \x in {a1,a2,a3,a4,a5,a6}{
\node[unode] at (\x) {};
}
\end{tikzpicture}
\caption{Friendship graph $T_3.$}
\label{fig8}
\end{figure}
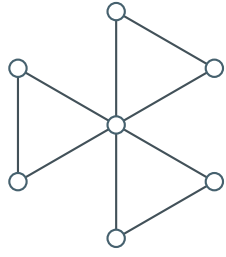

\begin{observation}\label{lemma5}
Let $T_n$ be a friendship graph with $n \geq 1$. Then $\chi_\rho(T_n)=n+2$.	
\end{observation}

\proof We know that $T_n$ with $n \geq 2$ has diameter $2$ and $\alpha(T_n)=n$. Then $\chi_\rho(T_n)=n+2$ by Lemma \ref{lemma4}. This rule trivially holds for $n=1$ since $T_1 \cong C_3$.
\qed

\begin{proposition}\label{pro10} 
Let $G \in \mathcal{G}_{1}^{(4)}(k_1,m_1)$ and $k_1, m_1 \geq 0$. Then
\[
\chi_\rho(G) =
\begin{cases}
3, & \text{if } m_1 = 0, \\[6pt]
m_1+2, & \text{if } m_1 \geq 1.
\end{cases}
\]
\end{proposition}

\proof 
Let $G \in \mathcal{G}_{1}^{(4)}(k_1,m_1)$.	For the first case suppose that  $m_1= 0$. Then $k_1\geq 1$.  Since $ C_4 \subseteq G$, we have \( \chi_\rho(C_4) \leq \chi_\rho(G) \). Therefore, since \( \chi_\rho(C_4) = 3 \), it follows that \( \chi_\rho(G) \geq 3 \). To show that \( \chi_\rho(G) = 3 \), it suffices to construct a feasible $3$-packing coloring of \( G \). Let $x_1$ be the cut vertex of $G$. Define a packing coloring $c:V(G)\to[3]$ as follows: $c(x_2) = c(x_4)=1$, $c(x_3)=2$, $c(x_1) = 3$, and  assign color $1$ to each leaf attached to $x_1$. Thus, we conclude that  $\chi_\rho(G) =  3$. 
	
For the remaining case suppose that  \( m_1 \geq 1 \). In this case, there exists a friendship graph $T_{m_1}$ with \( m_1 \geq 1 \), where $T_{m_1}$ is a subgraph of $G$ and by Observation \ref{lemma5}, we have $\chi_\rho(G) \geq \chi_\rho(T_{m_1}) = m_1+2$.
For the upper bound, we construct a  packing coloring of $G$ with $m_1+2$ colors.  
Let $x_1$ be the cut vertex of $G$. Since $m_1 \geq 1$, there exist $m_1$ number of $C_3$ blocks $(x_1,u_{1i},v_{1i})$, where $1 \leq i \leq m_1$. Define a packing coloring $c:V(G)\to[m_1+2]$ as follows: $c(x_2) = c(x_4)=1$, $c(x_3)=c(v_{11})=2$, $c(x_1) = 3$.
Moreover, assign color $1$ to all leaves attached to $x_1$, and all vertices $u_{1i}$.
Lastly, assign distinct colors $4, 5,\dots,m_1+2$ to the remaining $m_1-1$ vertices $v_{1i}$  for $i \geq 2$. Thus, we conclude that  $\chi_\rho(G) =  m_1+2$.
\qed

\begin{lemma}\label{pro11}
Let $G \in \mathcal{G}_{1}^{(4)}(k_1,m_1)$ with $k_1, m_1 \geq 0$. Then $G$ is not $\chi_{\rho}$-critical.	
\end{lemma}

\proof
Let $G \in \mathcal{G}_{1}^{(4)}(k_1,m_1)$.	Let further $m_1=0$. Then $k_1 \geq 1$, \ie, there exists at least one pendant edge at $x_1$. Consider such an edge $e$. Since $C_4$ is a subgraph of $G-e$, $\chi_{\rho}(G-e) \geq 3$. By Proposition \ref{pro10}, we also have $\chi_{\rho}(G)=3$. Therefore, $\chi_\rho(G-e)=\chi_\rho(G)$, and $G$ is not $\chi_\rho$-critical.
	
Now it remains to look at the case with $m_1 \geq 1$. Consider an edge $e$ contained in the  main block $C_4$. In this case $T_{m_1}$ is a subgraph of $G-e$, and $\chi_{\rho}(G-e) \geq m_1+2$ by Observation \ref{lemma5}. In addition, by Proposition \ref{pro10}, we also have $\chi_{\rho}(G)=m_1+2$. Therefore, $\chi_\rho(G-e)=\chi_\rho(G)$, and $G$ is not $\chi_\rho$-critical.
\qed

\begin{lemma}~{\rm \cite{bresar-2022}}\label{lemma6}
Let $G \cong \mathcal{G}_{2}^{(4)}(1,0;1,0)$. Then $G$ is $4$-$\chi_{\rho}$-critical.
\end{lemma}

\begin{proposition}\label{pro12}
Let $G \in \mathcal{G}_{2}^{(4)}(k_{1},m_{1};k_{2},m_{2})$ and $k_1, k_2 \geq 0, \; m_1, m_2 \geq 0$. Let further $t = m_1+m_2$. Then	
\[
\chi_\rho(G) =
\begin{cases}
4, & \text{if } t = 0, \\[6pt]
t+3, & \text{ otherwise}.
\end{cases}
\]
\end{proposition}

\proof
Let $G \in \mathcal{G}_{2}^{(4)}(k_{1},m_{1};k_{2},m_{2})$ and $t = m_1+m_2$. For the first case suppose that $m_1= m_2=0$. 
By the definition of $\mathcal{G}_{2}^{(4)}(k_{1},m_{1};k_{2},m_{2})$, we have $k_1, k_2 \geq 1$, \ie, each of the vertices $x_1$ and $x_2$ has at least one pendant edge attached to it. 
Hence, $G$ contains $\mathcal{G}_{2}^{(4)}(1,0;1,0)$ as an induced subgraph. 
Therefore, by Lemma~\ref{lemma6}, we have $\chi_\rho(G) \geq \chi_\rho\big(\mathcal{G}_{2}^{(4)}(1,0;1,0)\big) = 4.$  On the other hand, clearly the number of vertices of \(G\) is \( k_1 + k_2 + 4  \) and $\alpha(G)=k_1+k_2+1$. Therefore, $\chi_\rho(G) \leq 4$ by Lemma \ref{lemma4}. In this case, we conclude that  $\chi_\rho(G) =  4$.
	
For the remaining case suppose that $t\geq 1$. Then at least one of $m_1$ and $m_2$ is positive. By symmetry, we may assume that $m_1\ge 1$. Obviously $|V(G)|=2t+k_1+k_2+4$.
Note that any packing coloring of $G$ of size $\ell$ partitions $V(G)$ into $\ell$ color classes $X_1, X_2, \dots, X_{\ell}$ such that each $X_i$ is an $i$-packing for each $i \in \{1, 2, \dots, \ell\}$. So, $X_1$ contains at most $\alpha(G)=t+k_1+k_2+1$ vertices, $X_2$ contains at most $2$ vertices since there are no $3$ vertices whose pairwise distances are equal to $3$, and each $X_j$ contains exactly one vertex for each $j \in \{3, 4, \dots, \ell\}$ since $G$ has diameter $3$. 
Thus, we obtain the following inequality:
$$|V(G)|=|X_1| + \dots + |X_{\ell}| \leq \alpha(G) + 2 + (\ell - 2) = \alpha(G) + \ell.$$
By rearranging the terms, it follows that
\[
\ell \geq |V(G)| - \alpha(G)
= (2t + k_1 + k_2 + 4) - (t + k_1 + k_2 + 1)
= t + 3.
\]
Hence, we conclude that $\chi_{\rho}(G) \geq t + 3.$
For the upper bound, we construct a packing coloring of $G$ with $t+3$ colors. Let $x_1$ and $x_2$ be cut vertices of $G$. Since $ m_1 + m_2 \geq 1$, we consider the case where $m_1$ number of $C_3$ blocks $(x_1,u_{1i},v_{1i})$ for $1 \leq i \leq m_1$, and $m_2$ number of $C_3$ blocks $(x_2,u_{2j},v_{2j})$ for $1 \leq j \leq m_2$ exist. Define a packing coloring $c:V(G)\to[t+3]$ as follows. 
$c(x_4)=1, c(x_3)=c(v_{11})=2$, $c(x_1)=3, c(x_2)=4.$ 
Assign color $1$ to all leaves attached to $x_1$, $x_2$, and all vertices $u_{1i}, u_{2j}$.
Lastly, if there exists, assign distinct colors $5,6,\dots,t+3$ to the remaining $t-1$ vertices $v_{1i}$ and $v_{2j}$. Thus, we conclude that  $\chi_\rho(G) =  t+3$. 
\qed

\begin{proposition}\label{pro15}
Let $G \in \mathcal{G}_{2}^{(4)}(k_{1},m_{1};k_{2},m_{2})$. If $m_1=m_2=0$ and $k_1+k_2 \geq 3$, then $G$ is not $\chi_{\rho}$-critical. 
\end{proposition}

\proof
Suppose for the sake of a contradiction that $G \in \mathcal{G}_{2}^{(4)}(k_{1},m_{1};k_{2},m_{2})$, where $m_1=m_2=0$ and $k_1+k_2 \geq 3$, is a \( \chi_\rho \)-critical graph.  Without loss of generality, assume that $k_1 \geq k_2$, \ie, $k_1 \geq 2$. Consider a pendant edge $e$ incident to $x_1$. Then  $\chi_\rho(G-e)=\chi_\rho(\mathcal{G}^{(4)}_2(k_1-1,0; k_2,0)).$ By Proposition \ref{pro12}, we have $\chi_\rho(G)=\chi_\rho(G-e)=4$, contradicting the assumption that $G$ is $\chi_\rho$-critical.

\begin{lemma}\label{lemma7}
	For every graph $H \in \{\mathcal{H}(k_1, m_1; k_2, 0) \mid m_1 \geq 1, k_2\geq 2 \}$, $\chi_\rho(H) = |V(H)|-\alpha(H)+1.$
\end{lemma}

\proof
Let $H \in \{\mathcal{H}(k_1, m_1; k_2, 0) \mid m_1 \geq 1, k_2\geq 2 \}$. 
Since $\operatorname{diam}(H)=3$, any color class $X_i=\{v\in V(H): c(v)=i\}$ satisfies $|X_i|\le 1$ for each $i\ge3$. 
Hence, only the color classes $X_1$ and $X_2$ may contain more than one vertex. In addition, $X_2$ contains at most $2$ vertices since there are no $3$ vertices of $H$ whose pairwise distances are equal to $3$. Therefore, $\chi_\rho(H)=|V(H)|-|X_1|-|X_2|+2.$ If $|X_2|=1$, then $\chi_\rho(H)\ge |V(H)|-\alpha(H)+1$ since $|X_1|\le \alpha(H)$.
If $|X_2|=2$, then the two vertices in $X_2$ must be adjacent to different cut vertices of $H$, namely $u_1$ and $u_2$. One of these vertices must be a leaf connected to $u_2$, and the other one is either a leaf adjacent to $u_1$ or it is a vertex of a $C_3$ block connected to $u_1$. In either case, we have $|X_1|\le \alpha(H)-1$.
Consequently, $\chi_\rho(H)=|V(H)|-|X_1|\ge |V(H)|-\alpha(H)+1.$ On the other hand, in all cases by Lemma \ref{lemma4}, we have 
$\chi_\rho(H)\leq |V(H)|-\alpha(H)+1.$
Combining these bounds, we obtain
$\chi_\rho(H)=|V(H)|-\alpha(H)+1.$
\qed

\begin{proposition}\label{pro16}
Let $G \in \mathcal{G}_{2}^{(4)}(k_{1},m_{1};k_{2},m_{2})$. If $m_1+m_2 \geq 1$ and $k_1+k_2 \geq 1$, then $G$ is not $\chi_{\rho}$-critical. 
\end{proposition}

\proof
Suppose for the sake of a contradiction that $G \in \mathcal{G}_{2}^{(4)}(k_{1},m_{1};k_{2},m_{2})$, where $m_1+m_2 \geq 1$ and $k_1+k_2 \geq 1$, is a \( \chi_\rho \)-critical graph. 
For the first case suppose that exactly one of $m_1$ and $m_2$ is zero. By symmetry, let $m_2=0$. In this case, by the definition of $\mathcal{G}_{2}^{(4)}(k_{1},m_{1};k_{2},m_{2})$, $k_2 \ge 1$ and $m_1 \ge 1$. 
Consider the edge $e = x_3x_4$. 
By removing $e$, we obtain $G - e \cong \mathcal{H}(k_1 + 1, m_1; k_2 + 1, 0).$
By Lemma~\ref{lemma7}, we have $\chi_\rho(G - e) = m_1 + 3$. 
Moreover, by Proposition~\ref{pro12}, $\chi_\rho(G) = m_1 + 3$. 
Then $\chi_\rho(G - e) = \chi_\rho(G)$, which contradicts the assumption that $G$ is $\chi_\rho$-critical.

For the remaining case suppose that $m_1, m_2 \geq 1$. Let $t=m_1+m_2$. Since $k_1+k_2 \ge 1$, there is at least one pendant edge at $x_1$ or $x_2$, say $e$. 
Without loss of generality, assume it is at $x_1$.
Consider the graph $G-e$. Then $
\chi_\rho(G-e)=\chi_\rho(\mathcal{G}^{(4)}_2(k_1-1,m_1; k_2,m_2)).$ Clearly, by Proposition~\ref{pro12},  $\chi_\rho(G)= \chi_\rho(G-e) = t + 3$, which contradicts the assumption that $G$ is $\chi_\rho$-critical.

\begin{lemma}\label{pro13}
Let $G \in \mathcal{G}_{2}^{(4)}(k_{1},m_{1};k_{2},m_{2})$ and $k_1, k_2 \geq 0, \; m_1, m_2 \geq 0$.
Then $G$ is $\chi_\rho$-critical if and only if one of the following holds:
\begin{enumerate}
\item[(i)] $k_1=k_2=1$ and $m_1=m_2=0$,
\item[(ii)] $k_1=k_2=0$ and $m_1,m_2\ge 1$.
\end{enumerate}
\end{lemma}

\proof
Let $G \cong \mathcal{G}_{2}^{(4)}(k_{1},m_{1};k_{2},m_{2})$ and $t=m_1+m_2$. Let $G$ be $\chi_\rho$-critical. Then by the definition of $\mathcal{G}_{2}^{(4)}(k_{1},m_{1};k_{2},m_{2})$ together with Propositions~\ref{pro15} and~\ref{pro16}, we have either $(i)$ $m_1=m_2=0$ and $k_1=k_2=1$ or $(ii)$  $k_1=k_2=0$ and $m_1,m_2\ge 1$.
	
For the converse, suppose that either $(i)$ or $(ii)$ holds. 
If $(i)$ holds, then by Lemma~\ref{lemma6}, $G$ is $\chi_\rho$-critical. 
Now assume that $(ii)$ holds, \ie, $k_1 = k_2 = 0$ and $m_1, m_2 \ge 1$.
Since $\chi_\rho(G)=t+3$ by Proposition \ref{pro12},  we show that for every edge $e$ of $G$, we have $\chi_{\rho}(G-e) < \chi_{\rho}(G)=t+3$. Note that the edges of \(G\) can be classified into six distinct types according to their position in the structure of \(G\). Without less of generality, let $m_1 \geq m_2$.
	
\textbf{Case 1:} Consider the edge \( e = x_1x_2 \). In the graph \( G-e \), define a packing coloring \( c : V(G-e) \to [t+2] \) as follows:
$c(x_4) = 1, 
c(x_1) = c(x_2) = 2,
c(v_{11}) =  c(v_{21}) = 3,
c(x_3) = 4.$
Assign color \(1\) to all vertices \(u_{1i}\) and \(u_{2j}\) for \(1 \le i \le m_1\) and \(1 \le j \le m_2\), respectively. 
Finally, assign distinct colors \(5, 6, \dots, t+2\) to the remaining \(t-2\) vertices \(v_{1i}\) and \(v_{2j}\). Then \(\chi_\rho(G-e) \le t+2\).
	
\textbf{Case 2:} Consider the edge $e=x_2x_3$. In the graph \( G-e \), define a packing coloring \( c : V(G-e) \to [t+2] \) as follows:
$c(x_4)= 1, 
c(x_3)=c(v_{11})=c(v_{21})=2,
c(x_1) = 3,
c(x_2)=4.$
Assign color \(1\) to all vertices \(u_{1i}\) and \(u_{2j}\) for \(1 \le i \le m_1\) and \(1 \le j \le m_2\), respectively. 
Finally, assign distinct colors \(5, 6, \dots, t+2\) to the remaining \(t-2\) vertices \(v_{1i}\) and \(v_{2j}\). Then \(\chi_\rho(G-e) \le t+2\).
	
\textbf{Case 3:} Consider the edge $e=x_3x_4$. In the graph \( G-e \), define a packing coloring \( c : V(G-e) \to [t+2] \) as follows:
$c(x_3)=c(x_4)= 1, 
c(v_{11})=c(v_{21})=2,
c(x_1) = 3,
c(x_2) = 4.$
Assign color \(1\) to all vertices \(u_{1i}\) and \(u_{2j}\) for \(1 \le i \le m_1\) and \(1 \le j \le m_2\), respectively. 
Finally, assign distinct colors \(5, 6, \dots, t+2\) to the remaining \(t-2\) vertices \(v_{1i}\) and \(v_{2j}\). Then \(\chi_\rho(G-e) \le t+2\).
	
\textbf{Case 4:} Consider the edge $e=x_1x_4$. In the graph \( G-e \), define a packing coloring \( c : V(G-e) \to [t+2] \) as follows:
$c(x_3)= 1, 
c(x_4)=c(v_{11})=c(v_{21})=2,
c(x_2) = 3,
c(x_1)=4.$
Assign color \(1\) to all vertices \(u_{1i}\) and \(u_{2j}\) for \(1 \le i \le m_1\) and \(1 \le j \le m_2\), respectively. 
Finally, assign distinct colors \(5, 6, \dots, t+2\) to the remaining \(t-2\) vertices \(v_{1i}\) and \(v_{2j}\). Then \(\chi_\rho(G-e) \le t+2\).
	
\textbf{Case 5:} Consider the edge $e=x_1v_{11}$ (which is symmetric to any edge of the form $x_1u_{1i}$ for all $1\le i\le m_1$, $x_1v_{1i}$ for all $2\le i\le m_1$, $x_2u_{2j}$ or $x_2v_{2j}$ for all $1\le j\le m_2$). In the graph \( G-e \), define a packing coloring \( c : V(G-e) \to [t+2] \) as follows:
$c(x_3)= 1, 
c(x_4)=c(v_{11})=c(v_{21})= 2,
c(x_1)=3,
c(x_2)=4.$
Assign color \(1\) to all vertices \(u_{1i}\) and \(u_{2j}\) for \(1 \le i \le m_1\) and \(1 \le j \le m_2\), respectively. 
Finally, assign distinct colors \(5, 6, \dots, t+2\) to the remaining \(t-2\) vertices \(v_{1i}\) and \(v_{2j}\). Then \(\chi_\rho(G-e) \le t+2\).
	
\textbf{Case 6:} Consider the edge $e=u_{11}v_{11}$ (which is symmetric to the edge $e = u_{1i}v_{1i}$ for all $2 \leq i \leq m_1$ or $e = u_{2j}v_{2j}$ for all $1 \leq j \leq m_2$).  Clearly, $G-e$ is isomorphic to the graph $\mathcal{G}_{2}^{(4)}(2,m_{1}-1;0,m_{2})$ and $\chi_{\rho}(\mathcal{G}_{2}^{(4)}(2,m_{1}-1;0,m_{2})) = t+2$ for $t\geq1$ by Proposition \ref{pro12}. We conclude that  $\chi_\rho(G-e) =  t+2$.
	
By considering all possible cases, we conclude that \( G \) is \( \chi_{\rho} \)-critical.
\qed

%%%%%%%%%%%%%%%%%%%%%%%%%%%%%%%%%%%%%%%%%%%%%%%%%
\subsubsection{$\chi_{\rho}$-critical cactus graphs with radius 2 and diameter $3$, where $C_3$ is the main block}
%%%%%%%%%%%%%%%%%%%%%%%%%%%%%%%%%%%%%%%%%%%%%%%%%
In this section, we will consider cactus graphs with radius 2 and diameter $3$ with $C_3$ as their main block. Note that such graphs form a subclass of block graphs. Hence, the characterization of block graphs with diameter $3$ (and thus radius $2$) given in Lemma~\ref{lemma8} applies directly to the graphs considered in this section.

\begin{lemma}~{\rm \cite{bresar-2022}}\label{lemma8}
	Let $G$ be a block graph with $\mathrm{diam}(G) = 3$, and let $B$ be the block induced by the center of $G$ (called the central block, and every other block of $G$ is called a side block). The graph $G$ is $\chi_\rho$-critical if and only if one of the following three possibilities holds for the vertices of $B$.
    \begin{itemize}
	    \item[($a$)] All vertices in $V(B)$ have degree $|V(B)|$.
	
	\item[($b$)] All vertices in $V(B)$ have degree $|V(B)|+1$, and exactly $|V(B)|-1$ vertices of $B$ have two leaf neighbors.
	
	\item[($c$)] For each vertex $x \in V(B)$, at least one of the following three properties holds.

    \begin{itemize}
        \item[($c_1$)] $x$ belongs to at least one side block of order at least $4$, and does not have any leaf neighbor;
	
	\item[($c_2$)] $x$ belongs to at least two side blocks of order $3$, and does not have any leaf neighbor;
	
	\item[($c_3$)] $x$ has degree $|V(B)|+1$ and has two neighbors, which are both leaves; in addition, at least one vertex in $V(B)$ satisfies one of the properties ($c_1$) or ($c_2$).
    \end{itemize}
	\end{itemize}
\end{lemma}

\begin{lemma}\label{teo1}
Let $G$ be a cactus graph with $\mathrm{rad}(G) = 2$  and  $\mathrm{diam}(G) = 3$ whose main block is $C_{3}$. 
Then $G$ is $\chi_\rho$-critical if and only if $G$ is one of the following graphs:
\begin{enumerate}[label=(\roman*)]
\item $\mathcal{G}_{3}^{(3)}(1,0;\,1,0;\,1,0)$,
\item $\mathcal{G}_{3}^{(3)}(2,0;\,2,0;\,0,1)$,
\item $\{\mathcal{G}_{3}^{(3)}(0,m_{1};\,0,m_{2};\,0,m_{3}) \mid m_{1},m_{2},m_{3}\ge 2\}$,
\item $\{\mathcal{G}_{3}^{(3)}(0,m_{1};\,0,m_{2};\,2,0) \mid m_{1},m_{2}\ge 2\}$,
\item $\{\mathcal{G}_{3}^{(3)}(0,m_{1};\,2,0;\,2,0) \mid m_{1}\ge 2\}$,
\item $\mathcal{H}(2,0;\,0,1)$,
\item $\{\mathcal{H}(0,m_{1};\,0,m_{2}) \mid m_{1},m_{2}\ge 2\}$,
\item $\{\mathcal{H}(0,m_{1};\,2,0) \mid m_{1}\ge 2\}$.
\end{enumerate}
\end{lemma}

\proof 
Since the main block of $G$ is $C_3$, the central block $B$ of $G$ is either isomorphic to $C_3$ or to $K_2$. Assume first that $B\cong C_3$. Then $|V(B)|=3$.
If condition {\rm($a$)} of Lemma \ref{lemma8} holds, then every vertex of $B$ has degree $3$. Therefore,
$G\cong \mathcal{G}_{3}^{(3)}(1,0;\,1,0;\,1,0)$ and {\rm($i$)} holds.
If condition {\rm($b$)} holds, then every vertex of $B$ has degree $4$, and exactly two vertices of $B$ have two leaf neighbors. Hence,
$G\cong \mathcal{G}_{3}^{(3)}(2,0;\,2,0;\,0,1)$ and {\rm($ii$)} holds.
Now suppose that condition {\rm($c$)} holds. Since {\rm($c_1$)} is not valid, at least one vertex of $B$ satisfies {\rm($c_2$)}. If all three vertices satisfy {\rm($c_2$)}, then each vertex belongs to at least two side blocks of order $3$, and so
$G\in \{\mathcal{G}_{3}^{(3)}(0,m_{1};\,0,m_{2};\,0,m_{3}) \mid m_{1},m_{2},m_{3}\ge 2\}$. Thus, {\rm($iii$)} holds.
If exactly two vertices satisfy {\rm($c_2$)}, then the remaining vertex satisfies {\rm($c_3$)}, and thus
$G\in \{\mathcal{G}_{3}^{(3)}(0,m_{1};\,0,m_{2};\,2,0) \mid m_{1},m_{2}\ge 2\}$ and {\rm($iv$)} holds.
If exactly one vertex satisfies {\rm($c_2$)}, then the other two satisfy {\rm($c_3$)}, and hence,
$G\in \{\mathcal{G}_{3}^{(3)}(0,m_{1};\,2,0;\,2,0) \mid m_{1}\ge 2\}.$ Thus, {\rm($v$)} holds.

Next assume that $B\cong K_2$. Then $|V(B)|=2$. If condition {\rm($a$)} of Lemma \ref{lemma8} holds, then every vertex of $B$ has degree $2$. Therefore,
$G\cong P_4,$ which is not valid for our case since its main block is not $C_3$.
If condition {\rm($b$)} holds, then each vertex of $B$ has degree $3$, and exactly one vertex of $B$ has two leaf neighbors. Therefore, 
$G\cong \mathcal{H}(2,0;\,0,1)$ and {\rm($vi$)} holds.
Finally, suppose that condition {\rm($c$)} holds. Again, {\rm($c_1$)} is not valid, so at least one vertex of $B$ satisfies {\rm($c_2$)}. If both vertices satisfy {\rm($c_2$)}, then each belongs to at least two side blocks of order $3$, and we obtain
$G\in \{\mathcal{H}(0,m_{1};\,0,m_{2}) \mid m_{1},m_{2}\ge 2\}.$ Thus, {\rm($vii$)} holds.
If exactly one vertex satisfies {\rm($c_2$)}, then the other satisfies {\rm($c_3$)}, and thus
$G\in \{\mathcal{H}(0,m_{1};\,2,0) \mid m_{1}\ge 2\},$ {\rm($viii$)} holds.
\qed

%%%%%%%%%%%%%%%%%%%%%%%%%%%%%%%%%%%%%%%%%%%%%%%%%
\section{Summary of results}
\label{sec:mainthm}
%%%%%%%%%%%%%%%%%%%%%%%%%%%%%%%%%%%%%%%%%%%%%%%%%

We now establish our main result concerning $\chi_{\rho}$-critical cactus graphs with radius $2$ and diameter $3$. 

\begin{theorem}\label{teo4}
Let $G$ be a cactus graph with $\mathrm{rad}(G) = 2$  and  $\mathrm{diam}(G) = 3$. Then $G$ is $\chi_{\rho}$-critical if and only if it is isomorphic to one of the following graphs:
\begin{enumerate}[label=(\roman*)]
\item $P_{4}$,
\item $\{\mathcal{G}_{1}^{(5)}(0,m_{1}) \mid m_{1} \ge 2\},$
\item $\mathcal{G}_{2}^{(4)}(1,0;\,1,0)$,
\item $\{\mathcal{G}_{2}^{(4)}(0,m_{1};\,0,m_{2}) \mid m_{1},m_{2} \ge 1\},$
\item $\mathcal{G}_{3}^{(3)}(1,0;\,1,0;\,1,0)$,
\item $\mathcal{G}_{3}^{(3)}(2,0;\,2,0;\,0,1)$,
\item $\{\mathcal{G}_{3}^{(3)}(0,m_{1};\,0,m_{2};\,0,m_{3}) \mid m_{1},m_{2},m_{3} \ge 2\},$
\item $\{\mathcal{G}_{3}^{(3)}(0,m_{1};\,0,m_{2};\,2,0) \mid m_{1},m_{2} \ge 2\},$
\item $\{\mathcal{G}_{3}^{(3)}(0,m_{1};\,2,0;\,2,0) \mid m_{1} \ge 2\},$
\item $\mathcal{H}(2,0;\,0,1)$,
\item $\{\mathcal{H}(0,m_{1};\,0,m_{2}) \mid m_{1},m_{2} \ge 2\},$
\item $\{\mathcal{H}(0,m_{1};\,2,0) \mid m_{1} \ge 2\}.$
\end{enumerate}
\end{theorem}

\proof Let \( G \) be a  cactus graph with $\mathrm{rad}(G) = 2$  and  $\mathrm{diam}(G) = 3$ and assume that $G$ is a $\chi_{\rho}$-critical graph. We first consider the case where \( G \) is tree. In this case, we have \( G \cong P_4 \) by Lemma \ref{pro3}. Next, suppose that \( G \) contains at least one cycle. Let $C^*$ be the main block. Recall that $C^*$ is either $C_3$, $C_4$, or $C_5$ by Observation \ref{lem:mainblock}. 

We begin with a common argument for the cases where $C^* \cong C_r$ with $r \in \{4,5\}$.
Since $\operatorname{diam}(C^*)=2$ and $\operatorname{diam}(G)=3$, we have $G \not\cong C^*$. Hence, $G$ contains at least one outer block. Moreover, each outer block has diameter $1$, and so every outer block is isomorphic to either $K_2$ or $C_3$. We also claim that all outer blocks must be attached to $C^*$ either at a single vertex or at two consecutive vertices. Indeed, if two outer blocks are attached at vertices $x_i$ and $x_j$ with $d_{C^*}(x_i,x_j)=2$, then for vertices $u$ and $v$ in these blocks, we obtain $d_G(u,v) \ge 4,$
contradicting $\operatorname{diam}(G)=3$.
	
\textbf{Case 1:} Let \( C^* \cong C_5=(x_1, x_2, x_3, x_4, x_5) \).

Then $G \in \mathcal{G}_{1}^{(5)}(k_1,m_1)$ or $G \in \mathcal{G}_{2}^{(5)}(k_{1},m_{1};k_{2},m_{2})$ in this case. By Lemmas \ref{pro7} and  \ref{pro9}, we have $G \in \{\mathcal{G}_{1}^{(5)}(0,m_1) \mid m_1 \ge 2\}$.
	
\textbf{Case 2:} Let \( C^* \cong C_4=(x_1, x_2, x_3, x_4) \).

Then $G \in \mathcal{G}_{1}^{(4)}(k_1,m_1)$ or $G \in \mathcal{G}_{2}^{(4)}(k_{1},m_{1};k_{2},m_{2})$ in this case. By Lemmas \ref{pro11} and  \ref{pro13}, we have $G \cong \mathcal{G}_{2}^{(4)}(1,0;1,0)$ or  $G \in \{\mathcal{G}_{2}^{(4)}(0,m_1;0,m_2) \mid m_1, m_2 \ge 1\}$.
	
\textbf{Case 3:} Let \( C^* \cong C_3=(x_1, x_2, x_3) \). 

In this case,  $G$ is a block graph. Then by Lemma \ref{teo1}, $G$ is one of the following graphs: $\mathcal{G}_{3}^{(3)}(1,0;\,1,0;\,1,0)$, $\mathcal{G}_{3}^{(3)}(2,0;\,2,0;\,0,1)$,
$\{\mathcal{G}_{3}^{(3)}(0,m_{1};\,0,m_{2};\,0,m_{3}) \mid m_{1},m_{2},m_{3} \ge 2\},$
$\{\mathcal{G}_{3}^{(3)}(0,m_{1};\,0,m_{2};\,2,0) \mid m_{1},m_{2} \ge 2\},$
$\{\mathcal{G}_{3}^{(3)}(0,m_{1};\,2,0;\,2,0) \mid m_{1} \ge 2\},$
$\mathcal{H}(2,0;\,0,1)$,
$\{\mathcal{H}(0,m_{1};\,0,m_{2}) \mid m_{1},m_{2} \ge 2\},$
$\{\mathcal{H}(0,m_{1};\,2,0) \mid m_{1} \ge 2\}.$
	
The converse is clear by Lemmas \ref{pro3}, \ref{pro7}, \ref{lemma6}, \ref{pro13}, and \ref{teo1}.
\qed

In this paper, we characterized $\chi_\rho$-critical graphs of radius $1$ (see Observation~\ref{obs:complete} and Theorem~\ref{thm:12}). We also provided a complete structural characterization of $\chi_\rho$-critical cactus graphs with radius $2$ and diameter $2$ or $3$ (see Theorems~\ref{teo3} and~\ref{teo4}).
These results naturally lead to the following open problem. In particular, complete characterizations of $\chi_\rho$-critical trees and of cactus graphs with radius $2$ and diameter $4$ remain open.

%%%%%%%%%%%%%%%%%%%%%%%%%%%%%%%%%%%%%%%%%%%%%%%
\section*{Acknowledgements}
%%%%%%%%%%%%%%%%%%%%%%%%%%%%%%%%%%%%%%%%%%%%%%%

This work was supported by the Scientific and Technological Research Council of Türkiye (TÜBİTAK) under grant no. 124F114.

\iffalse
%%%%%%%%%%%%%%%%%%%%%%%%%%%%
\section*{Declaration of interests}
%%%%%%%%%%%%%%%%%%%%%%%%%%%%

The authors declare that they have no known competing financial interests or personal relationships that could have appeared to influence the work reported in this paper.

%%%%%%%%%%%%%%%%%%%%%%%%%%%%
\section*{Data availability}
%%%%%%%%%%%%%%%%%%%%%%%%%%%%

Our manuscript has no associated data.
\fi

\baselineskip13pt
%%%%%%%%%%%%%%%%%%%%%%%%%%%%%%%%%%%%%%%%%%%%%%%%%%

\end{document}